\documentclass[11pt]{article}

\usepackage[margin=1in]{geometry}
\usepackage{amsmath,amssymb,amsthm,mathtools}
\usepackage{enumitem}
\usepackage{xcolor}
\usepackage{url}
\usepackage[numbers,sort&compress]{natbib}
\usepackage{hyperref}
\usepackage{authblk}

\hypersetup{
  colorlinks=true,
  linkcolor=blue!50!black,
  citecolor=blue!50!black,
  urlcolor=blue!50!black
}

\newtheorem{theorem}{Theorem}[section]
\newtheorem{proposition}[theorem]{Proposition}
\newtheorem{lemma}[theorem]{Lemma}
\newtheorem{corollary}[theorem]{Corollary}

\newtheorem{remark}[theorem]{Remark}
\newtheorem{example}[theorem]{Example}

\DeclareMathOperator{\Res}{Res}
\DeclareMathOperator{\Spec}{Spec}
\DeclareMathOperator{\Reg}{Reg}
\DeclareMathOperator{\ord}{ord}
\DeclareMathOperator{\Gal}{Gal}
\DeclareMathOperator{\diag}{diag}
\DeclareMathOperator{\Fitt}{Fitt}

\newcommand{\Qp}{\mathbb{Q}_p}

\newcommand{\Fcal}{\mathcal{F}}
\newcommand{\Rcal}{\mathcal{R}}

\title{Iwasawa-Type Spectral Resultant Growth Laws for Grover Walks on Graph Towers}
\author[1]{Jir\^o Akahori\thanks{E-mail: \texttt{akahori@se.ritsumei.ac.jp}. ORCID: 0000-0001-5214-9585}}
\author[1]{Taro Hayashi\thanks{Corresponding author. E-mail: \texttt{haya4taro@gmail.com}. ORCID: 0009-0005-0145-8042}}
\author[2]{Ryoichi Suzuki\thanks{E-mail: \texttt{rsuzukimath@gmail.com}. ORCID: 0000-0001-9979-1882}}

\affil[1]{Department of Mathematical Sciences, College of Science and Engineering, Ritsumeikan University, 1-1-1 Noji-higashi, Kusatsu, Shiga 525-8577, Japan.}
\affil[2]{Department of Business Economics, School of Management, Tokyo University of Science, 1-11-2 Fujimi, Chiyoda, Tokyo 102-0071, Japan.}
\date{}

\begin{document}
\maketitle

\begin{abstract}
Let $X_0\leftarrow X_1\leftarrow\cdots$ be a $\mathbb Z_p^d$-tower of
finite graphs, and let $U_n$ be the Grover transition matrix on $X_n$.
We study Iwasawa-type $p$-adic growth laws for the polynomial spectral
quantities
\[
\det P(U_n),
\]
where $P(A)$ is a monic polynomial.  The basic object is the spectral
resultant
\[
\mathcal R_{X,P}(T)=\operatorname{Res}_A(\mathcal F_X(A,T),P(A)),
\]
where $\mathcal F_X(A,T)$ is the universal Grover--Ihara spectral
polynomial of the tower.  In the integral setting, this resultant generates
the zeroth Fitting ideal of a natural finite module over the Iwasawa algebra;
when the resultant is nonzero, this module is torsion.  The polynomial $P$
packages prescribed spectral values into a single spectral packet.  If $P$ is coprime to the Bass factor
$A^2-1$ and $\mathcal R_{X,P}$ does not vanish at torsion characters,
then $\det P(U_n)$ is nonzero for all $n$ and we prove a Cuoco--Monsky
type leading asymptotic formula for $v_p(\det P(U_n))$.  The leading terms are
given explicitly by the $\mu$- and $\lambda$-invariants of
$\mathcal R_{X,P}$, with a separate correction coming from the Bass
factor.  For $P(A)=A-a$, with $a\ne\pm1$ and $a$ not an eigenvalue at
any level, this recovers the leading invariants in the fixed non-eigenvalue formula
for Grover characteristic polynomials.

We also prove an equivariant factorization of spectral resultants for
finite connected $p$-group covers.  As a consequence, we obtain an unramified
equivariant Kida formula under explicit integrality and nonzero-resultant
assumptions.  Finally, when
$\gcd(P,A^2-1)=1$, we show that torsion zeros of $\mathcal R_{X,P}$ correspond exactly to
occurrences of roots of $P$ as Grover eigenvalues at finite levels.  The examples include the $K_3$-tower, non-abelian Heisenberg
$5$-group covers, and an explicit torsion-zero spectral packet.
\end{abstract}



\section{Introduction}\label{sec:introduction}

Iwasawa theory for graph towers studies the asymptotic behavior of graph
invariants along towers of finite covers
\[
X=X_0\longleftarrow X_1\longleftarrow X_2\longleftarrow\cdots
\]
whose Galois groups are quotients of $\mathbb Z_p^d$. In the unweighted
case, the primary invariant is the number of spanning trees. This
invariant plays the role of the class number in classical Iwasawa theory
\cite{Iwasawa72,CM81}. Analogous asymptotic problems also appear for
function fields and link covers \cite{Wan19,TU25}. We use standard
graph-theoretic conventions from Serre and standard algebraic graph theory
references \cite{Ser03,GR01,BR12}. This graph-theoretic Iwasawa theory was
initiated by Gonet and Valli\`eres and further developed in the
$\mathbb Z_p^d$ setting by DuBose--Valli\`eres and Kleine--M\"uller
\cite{Gon22,Val21,DV23,KM24}. Related computational and structural aspects
of graph towers were developed by McGown--Valli\`eres \cite{MV23,MV24},
Kida-type formulae for graph towers were studied by Ray--Valli\`eres and
Kataoka \cite{Kida80,RV25,Kat24}, and ramified graph towers were studied by
Gambheera--Valli\`eres \cite{GV24}. Weighted graph variants are controlled
by weighted Artin--Ihara $L$-functions and their three-term determinant
formulae, following the work of Mizuno--Sato and Sato on weighted
complexities and weighted zeta functions \cite{MS03,MS04,MS14,Sat07}; see
also the Ihara--Bass--Stark--Terras theory of graph zeta functions
\cite{Ihara66,Bass92,StarkTerras96,Ter11}, Northshield's note on graph zeta
functions \cite{Northshield98}, and \cite[Theorem~2.16]{AMT2026}. The
$p$-adic asymptotic input used below goes back to Cuoco--Monsky and Monsky
\cite{CM81,Mon81}; related multi-variable asymptotic phenomena also appear
in function-field and topological settings \cite{Wan19,TU25}.

We also mention recent work of Adachi--Mizuno--Murooka--Tateno
\cite{AMMT2026} on distributions of Iwasawa $\lambda$-invariants of
constant $\mathbb Z_p$-towers over supersingular isogeny graphs. Their
work fixes primes $r$ and $p$, lets the isogeny prime $\ell$ vary, and
studies the distribution of $\lambda$-invariants of graph-theoretic
complexity characteristic elements using newforms, Galois representations,
and the Chebotarev density theorem.  The present paper follows a complementary direction.  We fix a graph tower and study spectral resultants
attached to Grover transition matrices, in particular the quantities
$\det P(U_n)$. Thus the two settings concern different kinds of Iwasawa
invariants: complexity-type $\lambda$-invariants varying with the isogeny
prime on the one hand, and Grover spectral resultants on a fixed tower on
the other.

The purpose of this paper is to develop Iwasawa-type growth laws for the
spectral quantities $\det P(U_n)$ attached to Grover walks.  For
background on quantum walks and search algorithms, see Portugal's
monograph \cite{Por18}.

Let $U_n$ be the transition matrix
of the Grover walk on $X_n$. Adachi--Mizuno--Tateno developed Iwasawa
theory for weighted graphs and, in the final section of their paper,
applied it to discrete-time quantum walks, building on the relation between Grover walks and zeta functions \cite{KS12,EHSW06,AMT2026}. In particular,
their Theorem~7.3 proves an Iwasawa-type asymptotic formula for
\[
  v_p(\det(aI_{2l_n}-U_n))
\]
when $a$ is a fixed non-eigenvalue of $U_n$ at every level
\cite[Theorem~7.3]{AMT2026}.

This paper develops a polynomial spectral refinement of
\cite[Section~7]{AMT2026}.  We keep the characteristic variable $A$ in
the determinant formula underlying \cite[Theorem~7.3]{AMT2026}, and
obtain a factorization of the whole characteristic polynomial
$\det(AI-U_n)$. This allows us to replace the linear polynomial $A-a$ by
a monic polynomial $P(A)$, subject to $P(1)P(-1)\ne0$, and to study the
polynomial spectral quantity
\[
  \det P(U_n).
\]
The point of passing from $A-a$ to $P(A)$ is that it treats prescribed
spectral packets rather than a single spectral value.  In this form, the
same resultant organizes both the finite-valued asymptotic regime and the
torsion-zero regime where prescribed Grover eigenvalues occur at finite
levels.

Let $X$ be a finite connected symmetric digraph, and assume throughout
that
\[
 d_v:=\#\{e\in E(X):o(e)=v\}\ge 1
 \qquad(v\in V(X)).
\]
Let
\[
  \alpha:E(X)\longrightarrow \Gamma\simeq\mathbb Z_p^d
\]
be a voltage assignment.  We fix an algebraic closure
\(\overline{\Qp}\) of \(\Qp\) and extend \(v_p\) to
\(\overline{\Qp}\).  Let \(K\subset\overline{\Qp}\) be a finite
extension of \(\Qp\).
All torsion character values below are taken in this fixed algebraic closure,
not necessarily in \(K\).  We equip $X$ with the Grover weight
\[
  w(e)=\frac{2}{d_{o(e)}}.
\]
Let $\tau:\Gamma\to K[[T_1,\ldots,T_d]]^\times$ be the universal
character. 
We write
\[
  \Lambda_{K,d}=K[[T_1,\ldots,T_d]].
\]
For $\gamma=(\gamma_1,\ldots,\gamma_d)\in\Gamma\simeq\mathbb Z_p^d$, we set
\[
  \tau(\gamma)=(1+T_1)^{\gamma_1}\cdots(1+T_d)^{\gamma_d}
\]
where the powers are defined by the \(p\)-adic binomial expansion.
We define the universal weighted adjacency matrix
\[
W_{\tau,\alpha}=
\left(
\sum_{e:v_i\to v_j} w(e)\tau(\alpha(e))
\right)_{i,j}\in M_m(\Lambda_{K,d}).
\]
The universal Grover--Ihara spectral polynomial is
\[
\Fcal_X(A,T)=
\det\left(A^2I_m-AW_{\tau,\alpha}+D^W(X)-I_m\right).
\]
For the Grover weight, the sum of the outgoing weights at each vertex is
equal to $2$. Hence
\[
D^W(X)=\diag\left(\sum_{o(e)=v_i}w(e)\right)_{i=1}^m=2I_m,
\]
and therefore
\[
\Fcal_X(A,T)=\det\left((A^2+1)I_m-AW_{\tau,\alpha}\right).
\]

Let
\[
W_n=\mu_{p^n}(\overline{\Qp})
=\{\zeta\in\overline{\Qp}^\times\mid \zeta^{p^n}=1\}
\]
be the group of \(p^n\)-th roots of unity in \(\overline{\Qp}\), and set
\[
W=\bigcup_{n\ge0} W_n.
\]
We also write
\[
W_n^d=\underbrace{W_n\times\cdots\times W_n}_{d\text{ times}},
\qquad
W^d=\underbrace{W\times\cdots\times W}_{d\text{ times}}.
\]
For \(\zeta=(\zeta_1,\ldots,\zeta_d)\in W_n^d\), we write
\[
\zeta-1=(\zeta_1-1,\ldots,\zeta_d-1).
\]
Here
\[
  \widehat{\Gamma_n}=\mathrm{Hom}(\Gamma_n,\overline{\Qp}^\times)
\]
denotes the character group of \(\Gamma_n\). Since
\(\Gamma_n\simeq(\mathbb Z/p^n\mathbb Z)^d\), each
\(\psi\in\widehat{\Gamma_n}\) is determined by a point
\(\zeta_\psi=(\zeta_1,\ldots,\zeta_d)\in W_n^d\), via
\[
  \psi(g_1,\ldots,g_d)=\zeta_1^{g_1}\cdots\zeta_d^{g_d}.
\]

If $\psi\in\widehat{\Gamma_n}$ corresponds to
$\zeta_\psi\in W_n^d$, then
\[
W_{\tau,\alpha}(\zeta_\psi-1)=W_\psi,
\]
where
\[
W_\psi=
\left(
\sum_{e:v_i\to v_j}w(e)\psi(\alpha_n(e))
\right)_{i,j}.
\]

We write
\[
\chi(X)=|V(X)|-l
\]
for the Euler characteristic of \(X\), where \(l\) denotes the number of
unoriented edges of \(X\).
The starting point is the universal spectral factorization:
\[
\det(AI_{2lq_n}-U_n)=
(A^2-1)^{-q_n\chi(X)}
\prod_{\zeta\in W_n^d}\Fcal_X(A,\zeta-1).
\]
Although the right-hand side is written as a rational function, it is in
fact a polynomial in $A$. This identity follows from the weighted
Ihara--Grover determinant identity recorded in
Lemma~\ref{lem:weighted-ihara-grover}, and it is the basis for the
spectral resultant construction below.

For a monic polynomial $P(A)\in K[A]$, we define the spectral resultant
element by
\[
\Rcal_{X,P}(T)=\Res_A(\Fcal_X(A,T),P(A)).
\]
In the integral setting, Proposition~\ref{prop:resultant-module} realizes
this element as the generator of the zeroth Fitting ideal of a natural
spectral resultant module over the Iwasawa algebra.
We use the convention that if $f(A)=\prod_i(A-\alpha_i)$ is monic, then
$\Res_A(f,g)=\prod_i g(\alpha_i)$. Thus, if
$C_n(A)=\det(AI-U_n)$, then $\Res_A(C_n,P)=\det P(U_n)$. Under the
assumptions $P(1)P(-1)\ne0$ and
\[
\Rcal_{X,P}(\zeta-1)\ne0\qquad(\zeta\in W^d),
\]
we prove a Cuoco--Monsky type asymptotic formula for
$v_p(\det P(U_n))$. The case $P(A)=A-a$, with $a\ne\pm1$ and
$a\notin\Spec(U_n)$ for all $n$, recovers the leading invariants in the
fixed non-eigenvalue formula.

At the level of the main theorems, the objects considered here are
different from those in the weighted-complexity part of \cite{AMT2026}.
There the characteristic element is Laplacian-type and controls weighted
complexities of graph towers.  In the present paper the characteristic
variable of the Grover transition matrix is retained, and the spectral
resultant controls polynomial spectral packets of Grover eigenvalues.  In
particular, the unramified Kida formula below concerns the spectral
resultant module attached to $\det P(U_n)$, rather than the weighted
complexity characteristic element.

The resultant formalism separates two complementary regimes.  In the
torsion non-vanishing regime, Cuoco--Monsky asymptotics govern the
valuations of $\det P(U_n)$.  In the vanishing regime, torsion zeros of
$\Rcal_{X,P}$ record prescribed spectral occurrence at finite levels.
Thus the same resultant element controls finite-valued asymptotics and
records the corresponding spectral obstruction.

The second part of the paper concerns finite $p$-group covers
$\pi:Y\to X$. We prove a formal equivariant factorization of spectral
elements. If $G=\Gal(Y/X)$ and $K$ is a splitting field for $G$, then
\[
\Fcal_Y(A,T)=\prod_{\rho\in\widehat G}\Fcal_{X,\rho}(A,T)^{d_\rho}.
\]
Here $d_\rho=\dim\rho$, and $\Fcal_{X,\rho}$ is the spectral element
twisted by $\rho$. This formal identity is the key algebraic input for
the unramified equivariant Kida formula below: it is obtained from the
decomposition of the regular representation and gives an identity of
spectral elements before specialization at torsion characters. Under the
integrality hypotheses and the nonzero-resultant assumptions stated in
Theorem~\ref{thm:kida}, and assuming the equivalent $\mu=0$ condition,
it yields
\[
  \lambda(\Rcal_{Y,P})=[Y:X]\lambda(\Rcal_{X,P}).
\]
This gives the unramified equivariant Kida formula for spectral
resultants. It shows that the spectral resultant module and its
$\lambda$-invariant are compatible with finite $p$-group covers, including
non-abelian covers.
It differs from the weighted complexity Kida formula of
\cite[Theorem~4.1]{AMT2026}. In the weighted complexity setting, a
correction term appears when $d=1$ because the Laplacian-type
characteristic element has a forced zero at the trivial character. In the
spectral resultant setting, in the torsion non-vanishing regime, no such
forced zero occurs.

Finally, motivated by the torsion-zero question raised in
\cite[Remark~6.1]{AMT2026}, we interpret the failure of the torsion non-vanishing
hypothesis. After separating the Bass factor $A^2-1$, exceptional zeros
of $\Rcal_{X,P}$ at torsion points are precisely occurrences of roots of
$P$ as Grover eigenvalues at finite levels of the tower.

The main contributions of this paper are as follows.
\begin{enumerate}[label=(\roman*)]
\item We introduce the spectral resultant element $\Rcal_{X,P}$, realize it
as a Fitting generator of a natural spectral resultant module in the
integral setting, and prove a Cuoco--Monsky type leading asymptotic formula
for $v_p(\det P(U_n))$ under a torsion non-vanishing condition.
\item We use a formal equivariant factorization of spectral elements as the key
algebraic input for an unramified equivariant Kida formula for finite $p$-group covers, including
non-abelian covers.
\item We prove an exceptional-zero/eigenvalue-occurrence criterion: after
separating the Bass factor, torsion zeros measure prescribed Grover spectral
occurrence in the finite layers of the tower, with multiplicities given by
an explicit $P$-spectral multiplicity formula.
\end{enumerate}
The passage from \(A-a\) to a monic polynomial \(P(A)\) has three roles.
It packages all roots of \(P\) into a single power series, realizes that
power series as a Fitting generator of the spectral resultant module, and
separates the torsion non-vanishing regime from the torsion-zero regime in
which prescribed Grover eigenvalues occur at finite levels.  These
features are not visible from a single specialization \(A=a\).
The examples at the end illustrate them through a linear non-vanishing
packet, a quadratic polynomial packet, and a torsion-zero packet.

Combining the weighted Ihara--Grover determinant identity, the leading
Cuoco--Monsky asymptotic theorem, and the fixed non-eigenvalue formula of
Adachi--Mizuno--Tateno, we construct the spectral resultant module, prove
the growth law for \(\det P(U_n)\), establish the equivariant factorization
and unramified Kida formula for spectral resultants, and identify
exceptional zeros with finite-level Grover spectral occurrence.

The paper is organized as follows. Section~\ref{sec:preliminaries} fixes
the notation for graph towers, Grover weights, Iwasawa invariants, and
resultants. Section~\ref{sec:universal-factorization} proves the universal
spectral factorization. Section~\ref{sec:resultant-iwasawa} introduces
spectral resultants and proves the growth formula for $\det P(U_n)$.
Section~\ref{sec:kida} proves the equivariant factorization and the unramified
quantum Kida formula. Section~\ref{sec:exceptional} interprets exceptional zeros as
eigenvalue occurrences after separating the Bass factor.
Section~\ref{sec:examples} gives the $K_3$-tower and non-abelian Heisenberg
examples.

\section{Preliminaries}\label{sec:preliminaries}

\subsection{Graphs, towers, and Grover weights}\label{subsec:graphs-towers}

Throughout the paper, $X$ is a finite connected symmetric digraph.
We fix an ordering
\[
  V(X)=\{v_1,\ldots,v_m\},
\]
and write
\[
  m=\#V(X),\qquad 2l=\#E(X),\qquad \chi(X)=m-l.
\]
Thus $l$ is the number of unoriented edges of $X$.
For an edge $e\in E(X)$, we write $o(e)$ and $t(e)$ for its origin and
terminus, and $\bar e$ for the inverse edge. For a vertex $v$, put
\[
 d_v=\#\{e\in E(X):o(e)=v\},
\]
and assume $d_v\ge1$ for every $v\in V(X)$. The Grover weight is
\[
 w(e)=\frac{2}{d_{o(e)}}.
\]
Let
\[
  \alpha:E(X)\to\Gamma\simeq\mathbb Z_p^d
\]
be a voltage assignment, with $\alpha(\bar e)=\alpha(e)^{-1}$ in
multiplicative notation. We use multiplicative notation for voltage
assignments and additive notation after identifying
$\Gamma\simeq\mathbb Z_p^d$.
For $n\ge0$, put
\[
\Gamma_n=\Gamma/p^n\Gamma,
\qquad
q_n=\#\Gamma_n=p^{dn}.
\]
Let
\[
\alpha_n:E(X)\longrightarrow\Gamma_n
\]
be the composition of $\alpha$ with the quotient map $\Gamma\to\Gamma_n$,
and set
\[
X_n=X(\Gamma_n,\alpha_n).
\]
Thus $X_n$ is the derived voltage graph with vertex set
$V(X)\times\Gamma_n$.  For each edge $e:v_i\to v_j$ of $X$ and each
$g\in\Gamma_n$, it has an edge
\[
(e,g):(v_i,g)\longrightarrow (v_j,g+\alpha_n(e)),
\]
where the addition is taken in $\Gamma_n$.

The Grover transition matrix on $X_n$ is denoted by $U_n$; thus $U_n$ has
size $2lq_n$. For a finite matrix $M$, we write $\Spec(M)$ for the set of
its eigenvalues over an algebraic closure; multiplicities are recorded
separately when needed.

We keep fixed an algebraic closure \(\overline{\Qp}\) of \(\Qp\) and a
fixed extension of \(v_p\) to \(\overline{\Qp}\). Let
\(K\subset\overline{\Qp}\) be a finite extension of \(\Qp\). For
\(n\ge0\), put
\[
W_n=\mu_{p^n}(\overline{\Qp}),\qquad W=\bigcup_{n\ge0}W_n.
\]
All torsion character values are taken in \(\overline{\Qp}\). We write
\[
\widehat{\Gamma_n}=\operatorname{Hom}(\Gamma_n,\overline{\Qp}^{\times}),
\]
and identify \(\widehat{\Gamma_n}\) with \(W_n^d\) via the chosen basis of
\(\Gamma\simeq\mathbb Z_p^d\).

\subsection{Universal characters and the Grover--Ihara spectral polynomial}\label{subsec:universal-characters}

Let
\[
\Lambda_{K,d}=K[[T_1,\ldots,T_d]].
\]
The universal character \(\tau:\Gamma\to\Lambda_{K,d}^{\times}\) is defined by
\[
\tau(\gamma_1,\ldots,\gamma_d)
=(1+T_1)^{\gamma_1}\cdots(1+T_d)^{\gamma_d},
\]
where the powers are defined by the \(p\)-adic binomial expansion.  Define
\[
W_{\tau,\alpha}=
\left(
\sum_{e:v_i\to v_j}w(e)\tau(\alpha(e))
\right)_{i,j}
\in M_m(\Lambda_{K,d}).
\]
The universal Grover--Ihara spectral polynomial is
\[
\Fcal_X(A,T)=
\det\left(A^2I_m-AW_{\tau,\alpha}+D^W(X)-I_m\right).
\]
For the Grover weight, \(D^W(X)=2I_m\), and hence
\[
\Fcal_X(A,T)=\det\left((A^2+1)I_m-AW_{\tau,\alpha}\right).
\]

\subsection{Resultants}\label{subsec:resultants}

For a monic polynomial $P$, we write $\Theta^{\mathrm{mult}}(P)$ for the
multiset of roots of $P$ in a splitting field, counted with multiplicity.
In Section~\ref{sec:exceptional}, we write $\Theta(P)$ for the set of
distinct roots and record their multiplicities separately.

For a monic polynomial $f(A)=\prod_i(A-\alpha_i)$, with roots counted with
multiplicity, we use the convention
\[
  \Res_A(f,g)=\prod_i g(\alpha_i).
\]
In particular, if $C(A)=\det(AI_N-M)$, then
\[
  \Res_A(C(A),P(A))=\det P(M)
\]
for every monic polynomial $P(A)$.

Since the leading term of $\Fcal_X(A,T)$ in $A$ is $A^{2m}$,
$\Fcal_X(A,T)$ is monic in $A$ of degree $2m$. We define
\[
\Rcal_{X,P}(T)=\Res_A(\Fcal_X(A,T),P(A)).
\]
For $P(A)=A-a$, this convention gives
\[
  \Rcal_{X,A-a}(T)=\Fcal_X(a,T).
\]
We call $\Rcal_{X,P}$ a spectral resultant element. It is a spectral
power-series invariant that plays the role of a characteristic power series
in the growth formula below. Proposition~\ref{prop:resultant-module} gives
a Fitting interpretation in the integral setting.

\begin{lemma}[Even-degree sign convention]\label{lem:even-degree}
Let $C(A)=\det(AI_N-M)$ be the characteristic polynomial of an
$N\times N$ matrix $M$, and let $P(A)\in K[A]$ be monic. Suppose that,
over a splitting field,
\[
  P(A)=\prod_{\theta\in\Theta^{\mathrm{mult}}(P)}(A-\theta),
\]
with roots counted with multiplicity. Then
\[
  \prod_{\theta\in\Theta^{\mathrm{mult}}(P)}C(\theta)
  =(-1)^{N\deg P}\det P(M).
\]
In particular, if $N$ is even, then
\[
  \prod_{\theta\in\Theta^{\mathrm{mult}}(P)}C(\theta)=\det P(M).
\]
\end{lemma}

\begin{proof}
Write $C(A)=\prod_{i=1}^N(A-\lambda_i)$ over an algebraic closure. Then
\[
\prod_{\theta}C(\theta)
=\prod_{\theta}\prod_i(\theta-\lambda_i)
=(-1)^{N\deg P}\prod_i\prod_{\theta}(\lambda_i-\theta)
=(-1)^{N\deg P}\prod_iP(\lambda_i).
\]
The last product is $\det P(M)$.
\end{proof}

\begin{proposition}[Spectral resultant module]\label{prop:resultant-module}
Let $\mathcal O$ be the valuation ring of a finite extension of
$\mathbb Q_p$, and put
\[
\Lambda_{\mathcal O,d}=\mathcal O[[T_1,\ldots,T_d]].
\]
Assume that
\[
P(A)\in\mathcal O[A]
\]
is monic and that
\[
\Fcal_X(A,T)\in\Lambda_{\mathcal O,d}[A].
\]
Let
\[
N_{X,P}=\Lambda_{\mathcal O,d}[A]/(P(A)).
\]
Multiplication by $\Fcal_X(A,T)$ defines a
$\Lambda_{\mathcal O,d}$-linear endomorphism
\[
m_{\Fcal_X}:N_{X,P}\longrightarrow N_{X,P}.
\]
Set
\[
M_{X,P}=\operatorname{coker}(m_{\Fcal_X}).
\]
Then
\[
\Fitt_0^{\Lambda_{\mathcal O,d}}(M_{X,P})=(\Rcal_{X,P}(T)).
\]
In particular, if $\Rcal_{X,P}(T)\ne0$, then $M_{X,P}$ is a finitely
generated torsion $\Lambda_{\mathcal O,d}$-module.
\end{proposition}

\begin{proof}
Since $P(A)$ is monic, $N_{X,P}$ is a finite free
$\Lambda_{\mathcal O,d}$-module of rank $\deg P$.  For the cokernel of an
endomorphism of a finite free module, the zeroth Fitting ideal is generated
by the determinant of a representing matrix.  This determinant is the
determinant of multiplication by $\Fcal_X(A,T)$ on
$\Lambda_{\mathcal O,d}[A]/(P(A))$, hence equals
\[
\Res_A(P,\Fcal_X).
\]
By the resultant convention and the even-degree sign identity,
\[
\Res_A(P,\Fcal_X)=(-1)^{\deg P\cdot \deg_A\Fcal_X}\Res_A(\Fcal_X,P)=\Rcal_{X,P}(T),
\]
because $\deg_A\Fcal_X=2m$ is even.  This proves the Fitting ideal
identity. If $\Rcal_{X,P}\ne0$, the endomorphism has nonzero determinant
over the domain $\Lambda_{\mathcal O,d}$, so its cokernel is torsion.
\end{proof}

\subsection{Iwasawa invariants and Cuoco--Monsky asymptotics}\label{subsec:iw-invariants}

Let $\mathcal O=\mathcal O_K$ be the valuation ring of $K$, let $\varpi$
be a uniformizer, let $k=\mathcal O/\varpi\mathcal O$ be the residue
field, and let $e(K/\mathbb Q_p)$ be the ramification index. Throughout,
for an element with coefficients in $\mathcal O_K$, an overline denotes
its reduction modulo the maximal ideal.
Let
\[
  F\in K[[T_1,\ldots,T_d]]
\]
be nonzero. We say that $F$ is admissible if there exist an integer $N(F)$
and an element $F_0\in\mathcal O[[T_1,\ldots,T_d]]$ such that
\[
  F=\varpi^{N(F)}F_0,
  \qquad \varpi\nmid F_0.
\]
Then we define the \(\mu\)-invariant of \(F\) by
\[
  \mu(F)=\frac{N(F)}{e(K/\mathbb Q_p)}.
\]
Let $\overline F_0$ be the reduction of $F_0$ in
$k[[T_1,\ldots,T_d]]$. The $\lambda$-invariant of $F$ is
\[
  \lambda(F)=\sum_{\mathfrak p}\ord_{\mathfrak p}(\overline F_0),
\]
where the sum is over the distinct height-one prime ideals
\[
\mathfrak p=(\sigma-1),\qquad \sigma\in\Gamma\setminus\Gamma^p,
\]
which divide $\overline F_0$. Since $k[[T_1,\ldots,T_d]]$ is a unique
factorization domain and $\overline F_0\ne0$, only finitely many such
primes contribute.

For a nonzero element $\overline F\in k[[T_1,\ldots,T_d]]$, we also
write
\[
  \lambda_k(\overline F)=
  \sum_{\mathfrak p}\ord_{\mathfrak p}(\overline F),
\]
where the sum is over the same height-one prime ideals
$\mathfrak p=(\sigma-1)$ which divide $\overline F$. We use this
auxiliary notation in the proof of the equivariant Kida formula. If
$F\in\mathcal O[[T_1,\ldots,T_d]]$ and $\mu(F)=0$, then
$\lambda(F)=\lambda_k(\overline F)$.

\begin{lemma}[Finite coefficient extensions and coordinates]\label{lem:scalar-extension}
Let $K'/K$ be a finite extension. If an admissible nonzero series
$F\in K[[T_1,\ldots,T_d]]$ is viewed as an element of
$K'[[T_1,\ldots,T_d]]$, then its normalized $\mu$-invariant and its
$\lambda$-invariant are unchanged. The same is true after a
$\operatorname{GL}_d(\mathbb Z_p)$-change of coordinates on the Iwasawa
algebra.
\end{lemma}

\begin{proof}
Let $\varpi$ and $\varpi'$ be uniformizers of $K$ and $K'$, and let
$e(K'/K)$ be the relative ramification index. If
$F=\varpi^N F_0$ is the admissible factorization over $K$, then over
$K'$ the exponent of $\varpi'$ is multiplied by $e(K'/K)$, while the
absolute ramification index is multiplied by the same factor. Hence the
normalized value $N/e(K/\mathbb Q_p)$ is unchanged.

For $\lambda$, reduction modulo the maximal ideal replaces the residue
field $k$ by a finite extension $k'$. If
$\mathfrak p=(\sigma-1)$ with $\sigma\in\Gamma\setminus\Gamma^p$, then a
$\operatorname{GL}_d(\mathbb Z_p)$-change of variables sends $\sigma-1$
to $T_1$ up to a unit. Thus the corresponding order is the $T_1$-adic
order of the reduced series. This order is unchanged after extending the
coefficient field from $k$ to $k'$, because
$k'[[T_1,\ldots,T_d]]/(T_1)\simeq k'[[T_2,\ldots,T_d]]$ is an integral
domain. Thus each summand in the divisor-sum definition of $\lambda$ is
unchanged, and hence the whole $\lambda$-invariant is unchanged.
\end{proof}

We use the following leading form of the Cuoco--Monsky asymptotic theorem.
The main growth formula below uses this leading form.  Sharper lower-order
expansions, such as those used in \cite[Theorem~3.5]{AMT2026}, may be
substituted when available.

\begin{theorem}[Cuoco--Monsky, leading form]\label{thm:CM}
Let $F\in\Lambda_{K,d}$ be nonzero and admissible. Suppose that
\[
F(\zeta-1)\ne0
\qquad
(\zeta\in W^d\setminus\{(1,\ldots,1)\}).
\]
Then, as $n\to\infty$,
\[
\sum_{\zeta\in W_n^d\setminus\{(1,\ldots,1)\}}
 v_p(F(\zeta-1))
=
\mu(F)p^{dn}+\lambda(F)n p^{(d-1)n}+O(p^{(d-1)n}).
\]
When $d=1$, the error term is eventually constant: there exists
$\nu\in\mathbb Q$ such that, for all sufficiently large $n$,
\[
\sum_{\zeta\in W_n\setminus\{1\}}
 v_p(F(\zeta-1))
=
\mu(F)p^n+\lambda(F)n+\nu.
\]
\end{theorem}

\begin{corollary}[Including the trivial character]\label{cor:CM-trivial}
Let $F\in\Lambda_{K,d}$ be nonzero and admissible. Assume that
\[
  F(\zeta-1)\ne0\qquad(\zeta\in W^d).
\]
Then, as $n\to\infty$,
\[
\sum_{\zeta\in W_n^d}v_p(F(\zeta-1))
=
\mu(F)p^{dn}+\lambda(F)n p^{(d-1)n}+O(p^{(d-1)n}).
\]
When $d=1$, there exists $\nu\in\mathbb Q$ such that, for all sufficiently
large $n$,
\[
\sum_{\zeta\in W_n}v_p(F(\zeta-1))
=
\mu(F)p^n+\lambda(F)n+\nu.
\]
\end{corollary}

\begin{proof}
The preceding theorem applies to the sum over
$W_n^d\setminus\{(1,\ldots,1)\}$. The omitted term is $v_p(F(0))$, which
is independent of $n$ and is absorbed into the $O(p^{(d-1)n})$ term.  In
the case $d=1$, it changes only the eventual constant.
\end{proof}

\section{Universal Spectral Factorization}\label{sec:universal-factorization}

In this section we prove the factorization of the whole characteristic
polynomial of the Grover walk on the finite layers of the tower.  We first
record the determinant identity for Grover walks that underlies this
factorization.

For $\psi\in\widehat{\Gamma_n}$, set
\[
W_\psi=
\left(
\sum_{e:v_i\to v_j} w(e)\psi(\alpha_n(e))
\right)_{i,j}.
\]
We define the twisted weighted Ihara determinant factor attached to $X$ by
\[
h_X(\psi,t,\alpha_n)=
\det\left(I_m-tW_\psi+t^2(D^W(X)-I_m)\right).
\]

\begin{lemma}[Weighted Ihara--Grover determinant identity]\label{lem:weighted-ihara-grover}
For every $n\ge0$, one has
\[
\det(I_{2lq_n}-tU_n)
=
(1-t^2)^{-q_n\chi(X)}
\prod_{\psi\in\widehat{\Gamma_n}} h_X(\psi,t,\alpha_n).
\]
The identity is first obtained after adjoining the values of the finite
characters \(\psi\), and the full product is Galois invariant; hence it
belongs to \(K(t)\).
\end{lemma}

\begin{proof}
Equip all finite layers with the pulled-back Grover weight
$w(e)=2/d_{o(e)}$. Here $\zeta(X_n,t)$ denotes the Mizuno--Sato weighted
zeta function associated with this pulled-back weight, and
$L(X,t,\psi,\alpha_n)$ denotes the corresponding weighted Artin--Ihara
$L$-function. By the Konno--Sato determinant relation \cite{KS12}, the
Grover characteristic determinant is identified with the weighted Ihara
zeta function:
\[
\det(I_{2lq_n}-tU_n)=\zeta(X_n,t)^{-1}.
\]
Since \(\Gamma_n\) is finite abelian, its regular representation decomposes
as the direct sum of the characters in \(\widehat{\Gamma_n}\). The
corresponding change of basis is the finite Fourier transform. It is
constant, independent of $t$, and it block-diagonalizes the weighted
adjacency matrix of $X_n$ into the blocks $W_\psi$.
For the abelian voltage cover $X_n/X$ with group $\Gamma_n$, the Artin
factorization for Mizuno--Sato weighted $L$-functions and the three-term
determinant formula \cite{MS04,Sat07,AMT2026} give
\[
\zeta(X_n,t)^{-1}
=
\prod_{\psi\in\widehat{\Gamma_n}} L(X,t,\psi,\alpha_n)^{-1}
=
(1-t^2)^{-q_n\chi(X)}
\prod_{\psi\in\widehat{\Gamma_n}} h_X(\psi,t,\alpha_n),
\]
where we use \(\#\widehat{\Gamma_n}=\#\Gamma_n=q_n\).
The preceding display is obtained over a finite extension of $K$ containing
the values of all characters of $\Gamma_n$. Since the product is taken over
all such characters, it is invariant under the corresponding Galois action
on these values and hence descends to $K(t)$. This proves the identity.
\end{proof}

\begin{lemma}\label{lem:h-specialization}
For every $\psi\in\widehat{\Gamma_n}$,
\[
A^{2m}h_X(\psi,A^{-1},\alpha_n)=\Fcal_X(A,\zeta_\psi-1).
\]
\end{lemma}

\begin{proof}
Substituting $t=A^{-1}$ gives
\[
h_X(\psi,A^{-1},\alpha_n)
=A^{-2m}\det\left(A^2I_m-AW_\psi+D^W(X)-I_m\right).
\]
Since $W_\psi=W_{\tau,\alpha}(\zeta_\psi-1)$, the determinant on the right
is $\Fcal_X(A,\zeta_\psi-1)$.
\end{proof}

\begin{theorem}[Universal spectral factorization]\label{thm:universal-factorization}
For every $n\ge0$, one has
\[
\det(AI_{2lq_n}-U_n)
=
(A^2-1)^{-q_n\chi(X)}
\prod_{\zeta\in W_n^d}\Fcal_X(A,\zeta-1).
\]
\end{theorem}

\begin{proof}
By Lemma~\ref{lem:weighted-ihara-grover},
\[
\det(I_{2lq_n}-tU_n)
=
(1-t^2)^{-q_n\chi(X)}
\prod_{\psi\in\widehat{\Gamma_n}}
 h_X(\psi,t,\alpha_n).
\]
Put $A=t^{-1}$.
\[
\det(AI_{2lq_n}-U_n)=
A^{2lq_n}\det(I_{2lq_n}-A^{-1}U_n).
\]
Thus
\[
\det(AI_{2lq_n}-U_n)
=
A^{2lq_n}(1-A^{-2})^{-q_n\chi(X)}
\prod_{\psi\in\widehat{\Gamma_n}}h_X(\psi,A^{-1},\alpha_n).
\]
Since
\[
1-A^{-2}=A^{-2}(A^2-1),
\]
and by Lemma~\ref{lem:h-specialization}, the right-hand side equals
\[
A^{2lq_n+2q_n\chi(X)-2mq_n}
(A^2-1)^{-q_n\chi(X)}
\prod_{\psi\in\widehat{\Gamma_n}}\Fcal_X(A,\zeta_\psi-1).
\]
Because $\chi(X)=m-l$, the exponent of $A$ is zero. 
Identifying
$\widehat{\Gamma_n}$ with $W_n^d$ gives the formula.
The equality is first obtained over a finite extension of $K$ containing
$W_n$. Since the product is taken over all elements of $W_n^d$, it is
invariant under the Galois action and hence descends to $K(A)$. Since the
left-hand side is the characteristic polynomial of $U_n$, the right-hand
side belongs to $K[A]$.
\end{proof}

\begin{corollary}[Multiplicity at a point]\label{cor:point-multiplicity}
For any $a\in\overline{\Qp}$,
\[
\ord_{A=a}\det(AI_{2lq_n}-U_n)
=
-q_n\chi(X)\ord_{A=a}(A^2-1)
+
\sum_{\zeta\in W_n^d}\ord_{A=a}\Fcal_X(A,\zeta-1).
\]
For $a=\pm1$, the identity is interpreted after scalar extension to
$\overline{\Qp}(A)$: the pole coming from the Bass factor is canceled by
zeros of the product term, since the left-hand side is a polynomial. In
particular, if $a\ne\pm1$, then the first term is zero.
\end{corollary}

\begin{proof}
Take the order at $A=a$ in Theorem~\ref{thm:universal-factorization}.
\end{proof}

\section{Spectral Resultant Iwasawa-Type Growth Formula}\label{sec:resultant-iwasawa}

In this section we prove an Iwasawa-type growth formula for
$v_p(\det P(U_n))$.  The proof proceeds by evaluating the universal
spectral factorization at the roots of $P$, which keeps the Bass factor
separate from the resultant operation.

\subsection{Root expression and admissibility}\label{subsec:root-admissibility}

\begin{lemma}\label{lem:R-root-expression}
Let $P(A)\in K[A]$ be monic and write
\[
P(A)=\prod_{\theta\in\Theta^{\mathrm{mult}}(P)}(A-\theta)
\]
over a splitting field, with roots counted with multiplicity. Then
\[
\Rcal_{X,P}(T)=\prod_{\theta\in\Theta^{\mathrm{mult}}(P)}\Fcal_X(\theta,T).
\]
In particular, for $P(A)=A-a$,
\[
\Rcal_{X,A-a}(T)=\Fcal_X(a,T).
\]
\end{lemma}

\begin{proof}
Since $\Fcal_X(A,T)$ is monic of even degree $2m$ in $A$, the sign in
the reflection identity for resultants is
\[
(-1)^{\deg_A\Fcal_X\cdot\deg P}=(-1)^{2m\deg P}=1.
\]
Hence
\[
\Res_A(\Fcal_X,P)=\Res_A(P,\Fcal_X).
\]
The assertion follows from the resultant convention.
\end{proof}

\begin{lemma}[Admissibility of spectral resultants]\label{lem:admissibility}
Let $P(A)\in K[A]$ be monic. Assume that all weights $w(e)$ and all
coefficients of $P$ lie in a finite extension $K/\mathbb Q_p$. If
$\Rcal_{X,P}(T)\ne0$, then $\Rcal_{X,P}$ is admissible. In particular,
$\mu(\Rcal_{X,P})$ and $\lambda(\Rcal_{X,P})$ are well-defined.
\end{lemma}

\begin{proof}
For $a\in\mathbb Z_p$, the binomial expansion $(1+T)^a$ lies in
$\mathbb Z_p[[T]]$. The same assertion in several variables follows by
taking products of the one-variable binomial expansions. Hence the universal
character $\tau(\gamma)$ lies in
$\mathbb Z_p[[T_1,\ldots,T_d]]$ for every $\gamma\in\Gamma$.

The entries of $W_{\tau,\alpha}$ are finite $K$-linear combinations of
such integral power series. Since only finitely many weights occur, there
is a single integer $N_0$ such that every coefficient of every entry of
$W_{\tau,\alpha}$ lies in $\varpi^{-N_0}\mathcal O_K$. Therefore, after
multiplying by a suitable power of $\varpi$, all entries of
\[
A^2I_m-AW_{\tau,\alpha}+D^W(X)-I_m
\]
belong to $\mathcal O_K[[T_1,\ldots,T_d]][A]$. The determinant is a
finite polynomial expression in these entries, so another single power
of $\varpi$ clears all denominators of all coefficients of $\Fcal_X(A,T)$.
Since $P$ has finitely many coefficients, their denominators are also
bounded.  The Sylvester matrix defining
$\Res_A(\Fcal_X(A,T),P(A))$ has finite size and entries with bounded
denominators. Hence a single further power of $\varpi$ clears all
denominators in the resultant. 
Since \(\Rcal_{X,P}\ne0\), and since its coefficients have
\(\varpi\)-adic valuations bounded below, let \(N\) be the minimum of the
valuations of its nonzero coefficients. Then
\[
\Rcal_{X,P}=\varpi^N R_0,
\qquad
R_0\in\mathcal O_K[[T_1,\ldots,T_d]],
\qquad
\varpi\nmid R_0.
\]
This is the required admissible factorization.
\end{proof}

\begin{theorem}[Iwasawa-type spectral resultant growth formula]\label{thm:resultant-iwasawa}
Let $P(A)\in K[A]$ be monic and assume
\[
P(1)P(-1)\ne0.
\]
Assume further that
\[
\Rcal_{X,P}(\zeta-1)\ne0\qquad(\zeta\in W^d).
\]
Then $\det P(U_n)\ne0$ for all $n\ge0$, and, as $n\to\infty$,
\[
 v_p(\det P(U_n))
 =
 \mu^{\mathrm{qw}}_{X,P}p^{dn}
 +
 \lambda^{\mathrm{qw}}_{X,P}n p^{(d-1)n}
 +
 O(p^{(d-1)n}).
\]
Here
\[
\mu^{\mathrm{qw}}_{X,P}
=
\mu(\Rcal_{X,P})
-
\chi(X)v_p\left(\Res_A(A^2-1,P)\right)
\]
and
\[
\lambda^{\mathrm{qw}}_{X,P}=\lambda(\Rcal_{X,P}).
\]
When $d=1$, there exists $\nu\in\mathbb Q$ such that, for all sufficiently
large $n$,
\[
 v_p(\det P(U_n))=
 \mu^{\mathrm{qw}}_{X,P}p^n+
 \lambda^{\mathrm{qw}}_{X,P}n+\nu.
\]
\end{theorem}

\begin{proof}
Let $C_n(A)=\det(AI_{2lq_n}-U_n)$ and $q_n=p^{dn}$. Let $L/K$ be a
splitting field of $P$. We use the fixed extension of the $p$-adic
valuation to $\overline{\mathbb Q}_p$, restricted to $L$. Write
\[
P(A)=\prod_{\theta\in\Theta^{\mathrm{mult}}(P)}(A-\theta),
\]
with roots counted with multiplicity.

Since $C_n(A)$ has even degree $2lq_n$, Lemma~\ref{lem:even-degree}
gives
\[
\det P(U_n)=\prod_{\theta\in\Theta^{\mathrm{mult}}(P)}C_n(\theta).
\]
By the universal spectral factorization, and because $P(1)P(-1)\ne0$,
each root $\theta$ satisfies $\theta\ne\pm1$, so
\[
C_n(\theta)
=
(\theta^2-1)^{-q_n\chi(X)}
\prod_{\zeta\in W_n^d}\Fcal_X(\theta,\zeta-1).
\]
Multiplying over all roots \(\theta\), using the resultant convention,
Lemma~\ref{lem:R-root-expression}, and the fact that \(\deg(A^2-1)=2\),
we obtain
\[
\det P(U_n)
=
\Res_A(A^2-1,P)^{-q_n\chi(X)}
\prod_{\zeta\in W_n^d}\Rcal_{X,P}(\zeta-1).
\]
The right-hand side is nonzero by the assumptions
$P(1)P(-1)\ne0$ and $\Rcal_{X,P}(\zeta-1)\ne0$ for all torsion
points. Thus $\det P(U_n)\ne0$ for all $n$. Taking $p$-adic valuations
gives
\[
 v_p(\det P(U_n))
=
-q_n\chi(X)v_p\left(\Res_A(A^2-1,P)\right)
+
\sum_{\zeta\in W_n^d}v_p(\Rcal_{X,P}(\zeta-1)).
\]
The torsion non-vanishing assumption implies in particular that
\(\Rcal_{X,P}\ne0\). By Lemma~\ref{lem:admissibility} and
Corollary~\ref{cor:CM-trivial},
\[
\sum_{\zeta\in W_n^d}v_p(\Rcal_{X,P}(\zeta-1))
=
\mu(\Rcal_{X,P})p^{dn}
+
\lambda(\Rcal_{X,P})n p^{(d-1)n}
+
O(p^{(d-1)n}).
\]
The Bass factor contributes
\[
-\chi(X)v_p\left(\Res_A(A^2-1,P)\right)p^{dn},
\]
which modifies only the leading $\mu$-coefficient. This proves the
formula.
\end{proof}

\subsection{Finite-level specializations of the spectral resultant module}\label{subsec:finite-level-module}

We first record a simple sufficient condition for the integrality
assumption used below.  If
\[
p\nmid d_v
\qquad
(v\in V(X)),
\]
then \(d_v\in \mathcal O_K^\times\), and hence the Grover weights satisfy
\[
w(e)=\frac{2}{d_{o(e)}}\in \mathcal O_K
\qquad
(e\in E(X)).
\]
It follows that the universal weighted adjacency matrix
\(W_{\tau,\alpha}\) has entries in
\(\mathcal O_K[[T_1,\ldots,T_d]]\).  Consequently,
\[
\mathcal F_X(A,T)
=
\det\left((A^2+1)I_m-AW_{\tau,\alpha}\right)
\in \mathcal O_K[[T_1,\ldots,T_d]][A].
\]
Thus, under the condition \(p\nmid d_v\) for all vertices \(v\), and for
any monic \(P(A)\in\mathcal O_K[A]\), the spectral resultant module
\(M_{X,P}\) is naturally defined over the integral Iwasawa algebra
\[
\Lambda_{\mathcal O_K,d}
=
\mathcal O_K[[T_1,\ldots,T_d]].
\]

\begin{proposition}[Finite-level length formula for the spectral resultant module]
\label{prop:finite-level-resultant-module-length}
Let \(P(A)\in\mathcal O_K[A]\) be monic and assume that
\[
\Fcal_X(A,T)\in\mathcal O_K[[T_1,\ldots,T_d]][A].
\]
For \(n\ge0\), put
\[
J_n=
\bigl((1+T_1)^{p^n}-1,\ldots,(1+T_d)^{p^n}-1\bigr)
\subset \Lambda_{\mathcal O_K,d},
\]
and set
\[
M_{X,P,n}
=
M_{X,P}\otimes_{\Lambda_{\mathcal O_K,d}}
\left(\Lambda_{\mathcal O_K,d}/J_n\right).
\]
Assume that
\[
\Rcal_{X,P}(\zeta-1)\ne0
\qquad
(\zeta\in W^d).
\]
Then \(M_{X,P,n}\) has finite \(\mathcal O_K\)-length for every \(n\ge0\),
and
\[
\frac{1}{e(K/\mathbb Q_p)}
\operatorname{length}_{\mathcal O_K}(M_{X,P,n})
=
\sum_{\zeta\in W_n^d}
v_p\bigl(\Rcal_{X,P}(\zeta-1)\bigr).
\]
\end{proposition}

\begin{proof}
By Proposition~\ref{prop:resultant-module}, \(M_{X,P}\) is the cokernel of
multiplication by \(\Fcal_X(A,T)\) on the finite free
\(\Lambda_{\mathcal O_K,d}\)-module
\[
N_{X,P}=\Lambda_{\mathcal O_K,d}[A]/(P(A)).
\]
Moreover,
\[
\det_{\Lambda_{\mathcal O_K,d}}(m_{\Fcal_X})=\Rcal_{X,P}(T).
\]
Tensoring with \(\Lambda_{\mathcal O_K,d}/J_n\) gives
\[
M_{X,P,n}
\simeq
\operatorname{coker}(m_{\Fcal_X,n}),
\]
where \(m_{\Fcal_X,n}\) is the induced endomorphism of
\[
N_{X,P}\otimes_{\Lambda_{\mathcal O_K,d}}
\Lambda_{\mathcal O_K,d}/J_n.
\]

Let
\[
R_n=\Lambda_{\mathcal O_K,d}/J_n.
\]
After putting \(S_i=1+T_i\), the quotient \(R_n\) is obtained by imposing
the monic equations \(S_i^{p^n}=1\).  Hence \(R_n\) is finite free over
\(\mathcal O_K\) of rank \(p^{dn}\).  In particular, the source and target
of \(m_{\Fcal_X,n}\) are finite free \(\mathcal O_K\)-modules.

As an \(R_n\)-linear endomorphism, \(m_{\Fcal_X,n}\) has determinant equal
to the image of \(\Rcal_{X,P}(T)\) in \(R_n\).  Hence, as an
\(\mathcal O_K\)-linear endomorphism, its determinant is the determinant of
multiplication by this image on the finite free \(\mathcal O_K\)-module
\(R_n\), equivalently the norm of this element from the finite
\(\mathcal O_K\)-algebra \(R_n\).

Let \(L/K\) be a finite extension containing \(W_n\).  Over \(L\), each
polynomial \(S_i^{p^n}-1\) splits into distinct linear factors
\[
S_i^{p^n}-1=\prod_{\zeta_i\in W_n}(S_i-\zeta_i).
\]
The Chinese remainder theorem gives
\[
R_n\otimes_{\mathcal O_K}L
\simeq
\prod_{\zeta\in W_n^d}L,
\]
where \(T_i=S_i-1\) is sent to \(\zeta_i-1\).  Under this decomposition,
the image of \(\Rcal_{X,P}(T)\) corresponds to
\[
\bigl(\Rcal_{X,P}(\zeta-1)\bigr)_{\zeta\in W_n^d}.
\]
Hence the determinant of multiplication by the image of \(\Rcal_{X,P}\)
becomes
\[
\prod_{\zeta\in W_n^d}\Rcal_{X,P}(\zeta-1)
\]
after base change to \(L\).  Although the individual factors lie in
\(L\), the full product is invariant under the Galois action on
\(W_n^d\), and hence lies in \(K\).  Thus
\[
\det_{\mathcal O_K}(m_{\Fcal_X,n})
=
\prod_{\zeta\in W_n^d}\Rcal_{X,P}(\zeta-1)
\]
as an element of \(K\).

By the torsion non-vanishing assumption, this determinant is nonzero.  Since
\(\mathcal O_K\) is a discrete valuation ring, the elementary
determinant-length formula for an endomorphism \(f\) of a finite free
\(\mathcal O_K\)-module gives
\[
\operatorname{length}_{\mathcal O_K}\operatorname{coker}(f)
=
 v_{\varpi}(\det f),
\]
where \(v_{\varpi}(\varpi)=1\).  Applying this to \(m_{\Fcal_X,n}\) and
using \(v_{\varpi}=e(K/\mathbb Q_p)v_p\), we obtain
\[
\frac{1}{e(K/\mathbb Q_p)}
\operatorname{length}_{\mathcal O_K}(M_{X,P,n})
=
\sum_{\zeta\in W_n^d}
v_p\bigl(\Rcal_{X,P}(\zeta-1)\bigr).
\]
\end{proof}

\begin{corollary}[Iwasawa-type growth of finite-level lengths]
\label{cor:finite-level-resultant-module-growth}
Under the assumptions of
Proposition~\ref{prop:finite-level-resultant-module-length}, as \(n\to\infty\),
one has
\[
\frac{1}{e(K/\mathbb Q_p)}
\operatorname{length}_{\mathcal O_K}(M_{X,P,n})
=
\mu(\Rcal_{X,P})p^{dn}
+
\lambda(\Rcal_{X,P})n p^{(d-1)n}
+
O(p^{(d-1)n}).
\]
When \(d=1\), there exists \(\nu\in\mathbb Q\) such that, for all
sufficiently large \(n\),
\[
\frac{1}{e(K/\mathbb Q_p)}
\operatorname{length}_{\mathcal O_K}(M_{X,P,n})
=
\mu(\Rcal_{X,P})p^n
+
\lambda(\Rcal_{X,P})n
+
\nu.
\]
\end{corollary}

\begin{proof}
The torsion non-vanishing assumption implies that
\[
\Rcal_{X,P}\ne0.
\]
By Lemma~\ref{lem:admissibility}, \(\Rcal_{X,P}\) is admissible. Hence
Corollary~\ref{cor:CM-trivial} gives
\[
\sum_{\zeta\in W_n^d}
v_p\bigl(\Rcal_{X,P}(\zeta-1)\bigr)
=
\mu(\Rcal_{X,P})p^{dn}
+
\lambda(\Rcal_{X,P})n p^{(d-1)n}
+
O(p^{(d-1)n}).
\]
The result follows from
Proposition~\ref{prop:finite-level-resultant-module-length}. The assertion
for \(d=1\) follows in the same way from the one-dimensional statement
in Corollary~\ref{cor:CM-trivial}.
\end{proof}

\begin{corollary}[Finite-level length and Grover determinants]
\label{cor:length-det-relation}
Assume, in addition to the assumptions of
Proposition~\ref{prop:finite-level-resultant-module-length}, that
\[
P(1)P(-1)\ne0.
\]
Then, for every \(n\ge0\),
\[
\frac{1}{e(K/\mathbb Q_p)}
\operatorname{length}_{\mathcal O_K}(M_{X,P,n})
=
v_p(\det P(U_n))
+
p^{dn}\chi(X)
v_p\left(\Res_A(A^2-1,P)\right).
\]
Thus the finite-level length recovers the Grover determinant valuation up
to the explicit Bass correction.
\end{corollary}

\begin{proof}
Let \(q_n=p^{dn}\). By the universal spectral factorization and the
resultant convention, as in the proof of
Theorem~\ref{thm:resultant-iwasawa}, we have
\[
\det P(U_n)
=
\Res_A(A^2-1,P)^{-q_n\chi(X)}
\prod_{\zeta\in W_n^d}
\Rcal_{X,P}(\zeta-1).
\]
Taking \(p\)-adic valuations gives
\[
v_p(\det P(U_n))
=
-q_n\chi(X)v_p\left(\Res_A(A^2-1,P)\right)
+
\sum_{\zeta\in W_n^d}
v_p\bigl(\Rcal_{X,P}(\zeta-1)\bigr).
\]
Using Proposition~\ref{prop:finite-level-resultant-module-length}, we obtain
\[
v_p(\det P(U_n))
=
-p^{dn}\chi(X)v_p\left(\Res_A(A^2-1,P)\right)
+
\frac{1}{e(K/\mathbb Q_p)}
\operatorname{length}_{\mathcal O_K}(M_{X,P,n}).
\]
Rearranging proves the assertion.
\end{proof}

\begin{proposition}[A unit criterion for non-vanishing]\label{prop:unit-criterion}
Assume that
\[
\Rcal_{X,P}(T)\in \mathcal O_K[[T_1,\ldots,T_d]]
\]
and that
\[
\Rcal_{X,P}(0)\in\mathcal O_K^\times.
\]
Then
\[
\Rcal_{X,P}(\zeta-1)\ne0\qquad(\zeta\in W^d).
\]
\end{proposition}

\begin{proof}
If the constant term \(\Rcal_{X,P}(0)\) is a unit, then
\(\Rcal_{X,P}\) is a unit in the formal power series ring
\(\mathcal O_K[[T_1,\ldots,T_d]]\). For any torsion point
\(\zeta\in W^d\), the specialization
\(T_i\mapsto \zeta_i-1\) is a continuous homomorphism to the valuation
ring of a finite extension of \(K\), and it sends units to units. Hence
\(\Rcal_{X,P}(\zeta-1)\ne0\).
\end{proof}

\begin{remark}
The condition \(\Rcal_{X,P}(0)\ne0\) means that no root of \(P\) occurs in
the base-level Grover spectrum, apart from the Bass factor already
excluded by \(P(1)P(-1)\ne0\). The stronger condition
\(\Rcal_{X,P}(0)\in\mathcal O_K^\times\) in
Proposition~\ref{prop:unit-criterion} is a useful sufficient condition for
the torsion non-vanishing hypothesis in
Theorem~\ref{thm:resultant-iwasawa}. It is not a necessary condition.
\end{remark}

\begin{example}[A non-vanishing case in the \(K_3\)-tower]\label{ex:k3-unit}
In the \(K_3\)-tower of Section~\ref{subsec:k3-tower}, one has
\[
\Rcal_{X,A-a}(T)=(a^3-(1+T))(a^3-(1+T)^{-1}).
\]
Thus
\[
\Rcal_{X,A-a}(0)=(a^3-1)^2.
\]
For \(p=2\) and \(a=2\), this constant term is
\[
(2^3-1)^2=49\in\mathbb Z_2^\times.
\]
Hence Proposition~\ref{prop:unit-criterion} applies and gives
\[
\Rcal_{X,A-2}(\zeta-1)\ne0
\qquad
(\zeta\in\mu_{2^\infty}).
\]
This agrees with the explicit computation in the same tower, where
\(v_2(\det(2I-U_n))=0\) at every level.
\end{example}

\begin{corollary}[The linear case]\label{cor:linear}
Let $a\in K$ satisfy $a\ne\pm1$, and assume that
\[
 a\notin\Spec(U_n)\qquad(n\ge0).
\]
Put $P(A)=A-a$. Then
\[
\Rcal_{X,P}(T)=\Fcal_X(a,T),
\qquad
\Res_A(A^2-1,P)=a^2-1.
\]
Hence
\[
\mu^{\mathrm{qw}}_{X,A-a}
=
\mu(\Fcal_X(a,T))-\chi(X)v_p(a^2-1),
\qquad
\lambda^{\mathrm{qw}}_{X,A-a}=\lambda(\Fcal_X(a,T)).
\]
Thus Theorem~\ref{thm:resultant-iwasawa} recovers the leading invariants in
\cite[Theorem~7.3]{AMT2026} in the special case $P(A)=A-a$.  If one uses
the sharper Cuoco--Monsky expansion employed in \cite[Theorem~3.5]{AMT2026},
the corresponding lower-order terms are recovered as well.
\end{corollary}

\begin{remark}[Meaning of the non-vanishing condition]
The condition
\[
\Rcal_{X,P}(\zeta-1)\ne0
\qquad
(\zeta\in W^d)
\]
means that $P(A)$ and $\Fcal_X(A,\zeta-1)$ have no common root for any
torsion character $\zeta$. In Section~\ref{sec:exceptional} we show that,
after separating the Bass factor $A^2-1$, the failure of this condition
is precisely the occurrence of a root of $P$ as a Grover eigenvalue at a
finite level of the tower.
\end{remark}

\subsection{A computational recipe}\label{subsec:computational-recipe}

We record the practical computation of the spectral resultant
$\Rcal_{X,P}$. First one forms the universal weighted adjacency matrix
$W_{\tau,\alpha}$ and computes
\[
\Fcal_X(A,T)=
\det(A^2I_m-AW_{\tau,\alpha}+D^W(X)-I_m).
\]
For a monic polynomial $P(A)$, the spectral resultant is obtained as the
determinant of the Sylvester matrix of $\Fcal_X(A,T)$ and $P(A)$ with
respect to the variable $A$ over the coefficient ring
$K[[T_1,\ldots,T_d]]$:
\[
\Rcal_{X,P}(T)=\Res_A(\Fcal_X(A,T),P(A)).
\]
Assume that \(\Rcal_{X,P}\) is nonzero. Since its coefficients have \(\varpi\)-adic valuations bounded below, 
let \(N\) be the minimum of the valuations of its nonzero coefficients. 
Then we can write 
\[ \Rcal_{X,P}=\varpi^N R_0, \qquad R_0\in\mathcal O_K[[T_1,\ldots,T_d]], \qquad \varpi\nmid R_0. \]
Then \[ \mu(\Rcal_{X,P})=\frac{N}{e(K/\mathbb Q_p)}, \]
and the $\lambda$-invariant is computed from the reduction $\overline R_0$
as the divisor sum over the height-one linear primes $(\sigma-1)$ in the
residue power series ring. If $\Rcal_{X,P}=0$, then these Iwasawa
invariants are not defined for this element. If $\Rcal_{X,P}\ne0$ but
$\Rcal_{X,P}(\zeta-1)=0$ for some torsion point $\zeta$, then the
exceptional-zero theory of Section~\ref{sec:exceptional} applies.  This
normalization is used in the examples below.

\section{Equivariant Factorization and the Unramified Quantum Kida Formula}
\label{sec:kida}

In this section we prove an equivariant factorization formula for
spectral elements and deduce a Kida-type formula for spectral resultant
elements.  The construction below is parallel to the Artin formalism for
Mizuno--Sato weighted $L$-functions.  Twisted matrices indexed by
irreducible representations are standard in the theory of weighted graph
zeta functions.  In the present spectral setting, the passage to the
Grover polynomial in $A$ allows the regular representation decomposition
to be used directly at the level of spectral elements.  This factorization
is the algebraic input for the unramified Kida formula below.  For finite
$p$-group covers, every irreducible representation has only trivial
Jordan--H\"older factors after reduction modulo $p$.  This gives the
congruence, used below, between the spectral resultant of the cover and a
power of the spectral resultant of the base graph.

\subsection{Equivariant spectral elements}\label{subsec:equivariant-spectral-elements}

Let
\[
\pi:Y\to X
\]
be a finite connected unramified Galois cover with Galois group
\[
G=\Gal(Y/X).
\]
We write
\[
Y=X(G,\beta)
\]
for a voltage description of the cover, where
\[
\beta:E(X)\to G.
\]
Let
\[
  \pi_E:E(Y)\longrightarrow E(X)
\]
be the map on directed edges induced by the covering map $\pi:Y\to X$.
In this section, we consider the pullback $\mathbb Z_p^d$-tower over $Y$
defined by the voltage assignment
\[
  \alpha\circ\pi_E:E(Y)\longrightarrow\Gamma.
\]
The weight on \(Y\) is the pullback weight 
\[ w(\tilde e)=w(\pi_E(\tilde e)). \]
Thus the spectral element of $Y$ is
\[
\Fcal_Y(A,T)
=
\det\left(
A^2I_{m|G|}
-
AW^Y_{\tau,\alpha\circ\pi_E}
+
D^W(Y)-I_{m|G|}
\right).
\]
Here
\[
W^Y_{\tau,\alpha\circ\pi_E}
=
\left(
\sum_{\tilde e:\tilde v_i\to\tilde v_j}
 w(\tilde e)\tau(\alpha(\pi_E(\tilde e)))
\right)_{i,j}.
\]
Before the product ordering $V(Y)=V(X)\times G$ is fixed, this is the
ordinary $m|G|\times m|G|$ weighted adjacency matrix of $Y$.

We order the vertices of $Y=X(G,\beta)$ as
\[
V(Y)=V(X)\times G.
\]
With respect to this ordering, our convention for the regular representation
gives the block expression
\[
W^Y_{\tau,\alpha\circ\pi_E}
=\left(
\sum_{e:v_i\to v_j} w(e)\tau(\alpha(e))\Reg(\beta(e))
\right)_{1\le i,j\le m}.
\]
The regular representation part of the contribution of an
edge \(e\) is \(\Reg(\beta(e))\).
If one uses the opposite row-column convention, then
\[
\Reg(\beta(e)^{-1})
\]
appears instead. This replaces every irreducible representation $\rho$ by
its dual representation $\rho^\vee$. Since $\rho\mapsto\rho^\vee$ is a
permutation of $\widehat G$ and $d_{\rho^\vee}=d_\rho$, the product
formula below is independent of this convention.

\subsection{Twisted spectral elements}\label{subsec:twisted-spectral-elements}

Let \(K/\mathbb Q_p\) be a finite extension containing the weights and
over which \(G\) splits. When a monic polynomial \(P(A)\) is fixed below,
we assume that its coefficients lie in \(K\).
Lemma~\ref{lem:scalar-extension} allows this finite scalar extension
without changing the normalized $\mu$-invariant or the $\lambda$-invariant
used below.
For $\rho\in\widehat G$, put $d_\rho=\dim\rho$. Define the $\rho$-twisted
weighted adjacency matrix by
\[
(W_{\rho,\tau,\alpha})_{ij}
=
\sum_{e:v_i\to v_j}
 w(e)\tau(\alpha(e))\rho(\beta(e)).
\]
This is an $m\times m$ block matrix whose entries are $d_\rho\times
d_\rho$ matrices.

We define the $\rho$-twisted spectral element by
\[
\Fcal_{X,\rho}(A,T)
=
\det\left(
A^2I_{md_\rho}
-
AW_{\rho,\tau,\alpha}
+
(D^W(X)-I_m)\otimes I_{d_\rho}
\right).
\]
For a monic polynomial $P(A)\in K[A]$, define
\[
\Rcal_{X,P,\rho}(T)
=
\Res_A
\bigl(\Fcal_{X,\rho}(A,T),P(A)\bigr).
\]
Likewise, set
\[
\Rcal_{Y,P}(T)=\Res_A(\Fcal_Y(A,T),P(A)).
\]

\begin{proposition}[Formal equivariant spectral factorization]\label{prop:equiv-factor}
With the notation above, one has
\[
\Fcal_Y(A,T)
=
\prod_{\rho\in\widehat G}
\Fcal_{X,\rho}(A,T)^{d_\rho}.
\]
Consequently, for every monic polynomial $P(A)\in K[A]$,
\[
\Rcal_{Y,P}(T)
=
\prod_{\rho\in\widehat G}
\Rcal_{X,P,\rho}(T)^{d_\rho}.
\]
\end{proposition}

\begin{proof}
Since the cover is unramified and the weight is pulled back, we have
\[
D^W(Y)=D^W(X)\otimes I_{|G|}.
\]
With the convention fixed above, the universal weighted adjacency matrix
of $Y$ is
\[W^Y_{\tau,\alpha\circ\pi_E}
=
\sum_{1\le i,j\le m}
\sum_{e:v_i\to v_j}
w(e)\tau(\alpha(e))
E_{ij}\otimes\Reg(\beta(e)).\]
Here $E_{ij}$ denotes the $m\times m$ matrix unit for the vertex
coordinate, while $\Reg(\beta(e))$ acts on the group-coordinate space
$K[G]$, whose basis is ordered by the group coordinate in
$V(Y)=V(X)\times G$.
Therefore the matrix defining $\Fcal_Y(A,T)$ is
\[
A^2I_{m|G|}
-
AW^Y_{\tau,\alpha\circ\pi_E}
+
(D^W(X)-I_m)\otimes I_{|G|}.
\]

Since $K$ is a splitting field for $G$, the regular representation
decomposes as
\[
\Reg\simeq\bigoplus_{\rho\in\widehat G}\rho^{\oplus d_\rho}.
\]
The corresponding change of basis is applied only to the group-coordinate
factor $K[G]$; hence it commutes with the $A^2$-term and with the
$(D^W(X)-I_m)$-term on the vertex coordinate. This change of basis is
constant, independent of $A$ and $T$, and therefore it preserves the
determinant as a power series in $A$ and $T$.  After this constant change
of basis, the above matrix becomes block diagonal. For each $\rho\in\widehat G$, the block
\[
A^2I_{md_\rho}
-
AW_{\rho,\tau,\alpha}
+
(D^W(X)-I_m)\otimes I_{d_\rho}
\]
appears $d_\rho$ times. Taking determinants gives
\[
\Fcal_Y(A,T)
=
\prod_{\rho\in\widehat G}
\Fcal_{X,\rho}(A,T)^{d_\rho}.
\]
The identity for resultants follows from multiplicativity of the
resultant.
\end{proof}

\subsection{Reduction of \texorpdfstring{$p$}{p}-group twists}\label{subsec:pgroup-reduction}

Assume from now on that $G$ is a finite $p$-group. Let
$\mathcal O_K$ be the valuation ring of $K$, let $\mathfrak m$
be its maximal ideal, and put $k=\mathcal O_K/\mathfrak m$.

We impose the integrality assumptions
\[
w(e)=\frac{2}{d_{o(e)}}\in\mathcal O_K\qquad(e\in E(X)),
\]
and
\[
P(A)\in\mathcal O_K[A]\quad\text{monic}.
\]

\begin{lemma}[Reduction of $p$-group twists]\label{lem:pgroup-reduction}
Let bars denote reduction modulo $\mathfrak m$. For every irreducible
$K$-representation $\rho$ of $G$, one has, in
$k[[T_1,\ldots,T_d]][A]$,
\[
\overline{\Fcal_{X,\rho}}(A,T)
=
\overline{\Fcal_X}(A,T)^{d_\rho}.
\]
Consequently, in $k[[T_1,\ldots,T_d]]$,
\[
\overline{\Rcal_{X,P,\rho}}(T)
=
\overline{\Rcal_{X,P}}(T)^{d_\rho}.
\]
\end{lemma}

\begin{proof}
Let $V_\rho$ be the representation space of $\rho$.  Choose an
arbitrary $\mathcal O_K$-lattice $L^{\mathrm{init}}\subset V_\rho$, and set
\[
L=\sum_{g\in G}\rho(g)L^{\mathrm{init}}.
\]
Then $L$ is a $G$-stable $\mathcal O_K$-lattice. Since $\mathcal O_K$ is a
discrete valuation ring, $L$ is free of rank $d_\rho$. After choosing an
$\mathcal O_K$-basis of $L$, we may assume that
\[
\rho(g)\in\operatorname{GL}_{d_\rho}(\mathcal O_K)
\qquad(g\in G).
\]

Reducing modulo $\mathfrak m$, we obtain a $k[G]$-module. The reduction
of an irreducible characteristic-zero representation need not be
irreducible. However, since $G$ is a finite $p$-group and $\operatorname{char} k=p$,
the group algebra $k[G]$ is local. Hence every simple $k[G]$-module is
trivial, and the reduced representation admits a composition series whose
quotients are all trivial.

Let
\[
0=M_0\subset M_1\subset\cdots\subset M_{d_\rho}=L/\mathfrak mL
\]
be a composition series of the reduced representation, with trivial
one-dimensional quotients.  The tensor product
$k^m\otimes_k(L/\mathfrak mL)$ is filtered by
$k^m\otimes_k M_j$.  With respect to a basis adapted to this filtration,
the matrix $\overline{W_{\rho,\tau,\alpha}}$ is upper block triangular
with $d_\rho$ diagonal blocks.  Each diagonal block acts on the vertex
coordinate and is equal to $\overline{W_{\tau,\alpha}}$. Therefore
\[
A^2I_{md_\rho}
-
A\overline{W_{\rho,\tau,\alpha}}
+
\overline{(D^W(X)-I_m)}\otimes I_{d_\rho}
\]
is upper block triangular with $d_\rho$ diagonal blocks, each equal to
\[
A^2I_m
-
A\overline{W_{\tau,\alpha}}
+
\overline{(D^W(X)-I_m)}.
\]
Taking determinants gives
\[
\overline{\Fcal_{X,\rho}}(A,T)
=
\overline{\Fcal_X}(A,T)^{d_\rho}.
\]

Since $P(A)$ is monic and integral, and since, with respect to the chosen
$G$-stable lattice,
\[
\Fcal_{X,\rho}(A,T)\in\mathcal O_K[[T_1,\ldots,T_d]][A],
\]
the Sylvester matrix defining the resultant has entries in
$\mathcal O_K[[T_1,\ldots,T_d]]$. Therefore reduction modulo $\mathfrak m$
commutes with taking the resultant:
\[
\overline{\Res_A(\Fcal_{X,\rho},P)}
=
\Res_A(\overline{\Fcal_{X,\rho}},\overline P).
\]
Using the first part of the lemma, already stated modulo $\mathfrak m$, and multiplicativity of the resultant,
we get
\[
\begin{aligned}
\overline{\Rcal_{X,P,\rho}}
&=\Res_A(\overline{\Fcal_X}^{\,d_\rho},\overline P) \\
&=\Res_A(\overline{\Fcal_X},\overline P)^{d_\rho} \\
&=\overline{\Rcal_{X,P}}^{\,d_\rho}.
\end{aligned}
\]
\end{proof}

\begin{theorem}[Unramified equivariant quantum Kida formula]\label{thm:kida}
Let $\pi:Y\to X$ be a finite connected unramified Galois cover with Galois group
$G$. Assume that $G$ is a finite $p$-group. Let $K/\mathbb Q_p$ be a
finite extension containing all weights and over which $G$ splits.
Let $P(A)\in\mathcal O_K[A]$ be monic.
Assume
\[
 w(e)=\frac{2}{d_{o(e)}}\in\mathcal O_K\qquad(e\in E(X)).
\]
Assume also that
\[
\Rcal_{X,P}\ne0,
\qquad
\Rcal_{Y,P}\ne0.
\]
Then
\[
\mu(\Rcal_{X,P})=0
\quad\Longleftrightarrow\quad
\mu(\Rcal_{Y,P})=0.
\]
Moreover, if this equivalent condition holds, then
\[
\lambda(\Rcal_{Y,P})=[Y:X]\lambda(\Rcal_{X,P}).
\]
\end{theorem}

\begin{proof}
By the integrality assumptions on the weights and on \(P\), the
resultants \(\Rcal_{X,P}\) and \(\Rcal_{Y,P}\) lie in
\(\mathcal O_K[[T_1,\ldots,T_d]]\).
By Proposition~\ref{prop:equiv-factor},
\[
\Rcal_{Y,P}=\prod_{\rho\in\widehat G}\Rcal_{X,P,\rho}^{d_\rho}.
\]
By Lemma~\ref{lem:pgroup-reduction},
\[
\overline{\Rcal_{X,P,\rho}}=\overline{\Rcal_{X,P}}^{d_\rho}.
\]
Therefore, in
\[
k[[T_1,\ldots,T_d]]
=\mathcal O_K[[T_1,\ldots,T_d]]/\mathfrak m,
\]
we have
\[
\overline{\Rcal_{Y,P}}
=
\prod_{\rho\in\widehat G}
\left(\overline{\Rcal_{X,P}}^{d_\rho}\right)^{d_\rho}
=
\overline{\Rcal_{X,P}}^{\sum_\rho d_\rho^2}
=
\overline{\Rcal_{X,P}}^{|G|}.
\]
Since $k[[T_1,\ldots,T_d]]$ is an integral domain,
\[
\overline{\Rcal_{Y,P}}\ne0
\quad\Longleftrightarrow\quad
\overline{\Rcal_{X,P}}\ne0.
\]
For an integral nonzero power series
$F\in\mathcal O_K[[T_1,\ldots,T_d]]$, we have
\[
\mu(F)=0
\quad\Longleftrightarrow\quad
\overline F\ne0.
\]
Indeed, $\mu(F)>0$ is equivalent to all coefficients of $F$ being
divisible by $\varpi$, which is equivalent to $\overline F=0$. Hence
\[
\mu(\Rcal_{X,P})=0
\quad\Longleftrightarrow\quad
\mu(\Rcal_{Y,P})=0.
\]
Assume now that this equivalent condition holds.  Since
$k[[T_1,\ldots,T_d]]$ is a unique factorization domain, the residue-field
divisor sum satisfies
\[
\lambda_k(\overline F^{\,r})=r\lambda_k(\overline F)
\]
for every nonzero $\overline F\in k[[T_1,\ldots,T_d]]$ and every
positive integer $r$.  Moreover, if $F\in\mathcal O_K[[T_1,\ldots,T_d]]$
and $\mu(F)=0$, then $\lambda(F)=\lambda_k(\overline F)$.  Therefore
\[
\lambda(\Rcal_{Y,P})
=
\lambda_k(\overline{\Rcal_{Y,P}})
=
\lambda_k(\overline{\Rcal_{X,P}}^{|G|})
=
|G|\lambda_k(\overline{\Rcal_{X,P}})
=
[Y:X]\lambda(\Rcal_{X,P}).
\]

\end{proof}

\begin{remark}[Algebraic and asymptotic \(\lambda\)-invariants]
The identity
\[
\lambda(\Rcal_{Y,P})=[Y:X]\lambda(\Rcal_{X,P})
\]
in Theorem~\ref{thm:kida} is an equality of algebraic invariants of
power series. It governs the leading asymptotics of
\(v_p(\det P(U_n))\) only in the torsion non-vanishing regime of
Theorem~\ref{thm:resultant-iwasawa}.  When torsion zeros occur, as in the torsion-zero packet in
Section~\ref{subsec:heisenberg-application}, the algebraic identity
persists, but the finite-valued growth formula is replaced by the
exceptional-zero interpretation of Section~\ref{sec:exceptional}.
\end{remark}

\begin{remark}[Connectedness]
The algebraic determinant identities and the equivariant factorization
above are formulated directly in terms of the covering data. In the usual
graph-Iwasawa theoretic setting one often works with connected towers, and
all examples below satisfy this condition.
\end{remark}

\begin{remark}[No $d=1$ correction]
In the weighted complexity Kida formula \cite[Theorem~4.1]{AMT2026}, the case $d=1$ has an additional
correction term. This comes from the forced zero at the trivial character
of the Laplacian-type characteristic element. In the torsion non-vanishing regime of Theorem~\ref{thm:resultant-iwasawa},
for instance when the unit criterion of Proposition~\ref{prop:unit-criterion}
applies, i.e. when $\Rcal_{X,P}(0)\in\mathcal O_K^\times$, there is no forced
trivial-character zero. Consequently the Kida formula takes the pure form
\[
\lambda(\Rcal_{Y,P})=[Y:X]\lambda(\Rcal_{X,P})
\]
also when $d=1$.
\end{remark}

\begin{remark}[Non-integral Grover weights]
The integrality condition in Theorem~\ref{thm:kida} is precisely the
condition needed for the mod $p$ reduction argument. The spectral
resultant itself may be formed over $\Lambda_{K,d}$ after enlarging $K$.
\end{remark}

\begin{remark}[Bass-corrected $\mu$-invariants]
Assume the hypotheses of Theorem~\ref{thm:kida}, and suppose in addition
that $P(1)P(-1)\ne0$ and $\mu(\Rcal_{X,P})=0$. Then the Bass-corrected
$\mu$-invariants also scale:
\[
\mu^{\mathrm{qw}}_{Y,P}=[Y:X]\mu^{\mathrm{qw}}_{X,P}.
\]
Indeed,
\[
\mu^{\mathrm{qw}}_{X,P}
=-\chi(X)v_p(\Res_A(A^2-1,P)),
\]
and $\chi(Y)=[Y:X]\chi(X)$.
\end{remark}

\section{Exceptional Zeros and Eigenvalue Occurrence}
\label{sec:exceptional}

In this section we explain the spectral meaning of the failure of the
torsion non-vanishing condition in Theorem~\ref{thm:resultant-iwasawa}. This is
a spectral-resultant analogue of the torsion-zero phenomenon mentioned in
\cite[Remark~6.1]{AMT2026}; in the spectral setting, such zeros have a direct
interpretation as prescribed Grover eigenvalue occurrences after
separating the Bass factor.

\subsection{\texorpdfstring{$P$}{P}-spectral multiplicities}\label{subsec:p-spectral-multiplicities}

Let $P(A)\in K[A]$ be monic. After replacing $K$ by a finite extension if
necessary, write
\[
P(A)=\prod_{\theta\in\Theta(P)}(A-\theta)^{e_\theta},
\]
where $\Theta(P)$ is the set of distinct roots of $P$. For a polynomial
$F(A)$, define the $P$-weighted order by
\[
\ord_P^*F=\sum_{\theta\in\Theta(P)}e_\theta\ord_{A=\theta}F(A).
\]
When $P$ is squarefree, this agrees with the unweighted sum over the roots
of $P$. Let
\[
C_n(A)=\det(AI_{2lq_n}-U_n).
\]
Define the $P$-spectral multiplicity at level $n$ by
\[
m_n(P)=\ord_P^*C_n(A)
=\sum_{\theta\in\Theta(P)}e_\theta\ord_{A=\theta}C_n(A).
\]
Thus $m_n(P)$ is the $P$-weighted total algebraic multiplicity of the
roots of $P$ among the eigenvalues of $U_n$.

\begin{theorem}[Multiplicity formula]\label{thm:multiplicity}
Let $P(A)\in K[A]$ be monic. Then
\[
 m_n(P)
 =
 -q_n\chi(X)\ord_P^*(A^2-1)
 +
 \sum_{\zeta\in W_n^d}\ord_P^*\Fcal_X(A,\zeta-1).
\]
In particular, if $\gcd(P,A^2-1)=1$, then
\[
 m_n(P)
 =
 \sum_{\zeta\in W_n^d}\ord_P^*\Fcal_X(A,\zeta-1).
\]
\end{theorem}

\begin{proof}
Work over a splitting field of \(P\). Taking the order at \(A=\theta\) in
the identity of Theorem~\ref{thm:universal-factorization}, multiplying by
\(e_\theta\), and summing over all \(\theta\in\Theta(P)\), gives
\[
 m_n(P)
 =
 -q_n\chi(X)\ord_P^*(A^2-1)
 +
 \sum_{\zeta\in W_n^d}\ord_P^*\Fcal_X(A,\zeta-1).
\]
If \(\gcd(P,A^2-1)=1\), then \(\ord_P^*(A^2-1)=0\), and the second formula
follows.
\end{proof}

\begin{theorem}[Exceptional zero and finite-level spectral occurrence]\label{thm:exceptional}
Assume that
\[
\gcd(P,A^2-1)=1.
\]
Then
\[
\det P(U_n)=0
\]
if and only if there exists $\zeta\in W_n^d$ such that
\[
\Rcal_{X,P}(\zeta-1)=0.
\]
\end{theorem}

\begin{proof}
The resultant $\Rcal_{X,P}(\zeta-1)$ vanishes if and only if $P(A)$ and
$\Fcal_X(A,\zeta-1)$ have a common root. Since $\gcd(P,A^2-1)=1$,
Theorem~\ref{thm:multiplicity} gives
\[
 m_n(P)=\sum_{\zeta\in W_n^d}\ord_P^*\Fcal_X(A,\zeta-1).
\]
Thus $m_n(P)>0$ if and only if such a $\zeta$ exists. On the other hand,
$m_n(P)>0$ is equivalent to the occurrence of at least one root of $P$ as
an eigenvalue of $U_n$, which is equivalent to $\det P(U_n)=0$.
\end{proof}

\begin{remark}
Theorem~\ref{thm:exceptional} refines the formal vanishing criterion for
resultants into a finite-level spectral statement. The universal
spectral factorization identifies each torsion specialization
$\Fcal_X(A,\zeta-1)$ with the corresponding contribution to the Grover
spectrum at level $n$.
\end{remark}

\begin{remark}[Conductors]
Under the coprimality assumption of Theorem~\ref{thm:exceptional}, if
\(\zeta\in W^d\) satisfies \(\Rcal_{X,P}(\zeta-1)=0\), then this zero
contributes at every level \(n\) such that \(\zeta\in W_n^d\). Thus a
torsion point \(\zeta\) does not determine a unique finite level. Rather,
it has a conductor and contributes at every level containing that torsion
character.
\end{remark}

\begin{remark}[Periodic spectra]
For periodic spectra one may consider $P(A)=A^r-1$. Since the factors
$A-1$ and $A+1$ interact with the Bass factor $A^2-1$, it is often more
natural to use
\[
P_r^{\mathrm{red}}(A)=\frac{A^r-1}{\gcd(A^r-1,A^2-1)}.
\]
Explicitly,
\[
\gcd(A^r-1,A^2-1)=
\begin{cases}
A^2-1, & r \text{ even},\\
A-1, & r \text{ odd}.
\end{cases}
\]
Then $\gcd(P_r^{\mathrm{red}},A^2-1)=1$, and
Theorem~\ref{thm:exceptional} applies directly.
\end{remark}

\begin{remark}[Torsion-flat structure]
By Monsky's theorem on torsion zeros \cite[Theorem~5.6]{Mon81}, the
torsion zero set of each nonzero power series
\(\partial_A^j\Fcal_X(\theta,T)\) is a finite union of torsion-translated
\(p\)-adic flats.

More precisely, one applies Monsky's theorem to each derivative
\(\partial_A^j\Fcal_X(\theta,T)\) and then takes finite intersections of
the resulting finite unions of flats. If one of the derivatives is
identically zero, its zero set is all of \(W^d\), which is itself regarded
as a flat.  This describes the possible shape of higher-multiplicity
exceptional loci.
\end{remark}

\section{Examples}\label{sec:examples}

In this section we discuss two families of examples. The first family
revisits the $K_3$-tower and shows that the spectral resultant formalism
recovers the cyclotomic structure of the Grover spectrum. The second family
gives concrete applications of the unramified equivariant quantum Kida formula to a
non-abelian $p$-group cover.

\subsection{The \texorpdfstring{$K_3$}{K3}-tower and cyclotomic spectra}\label{subsec:k3-tower}

We first revisit \cite[Example~28]{AMT2026}. The computation there
corresponds to the linear polynomial $P(A)=A-a$; here we keep $A$ as a
variable and obtain a formula for arbitrary monic $P(A)$.

Let $p=2$ and $d=1$. 
In this example, our notation gives
\[ W_n=\mu_{2^n}, \qquad W=\bigcup_{n\ge0}W_n=\mu_{2^\infty}. \]
We consider the base graph of the $K_3$-tower from
\cite[Example~28]{AMT2026}. Put $z=1+T$. For this tower, the spectral
factor is
\[
\Fcal_X(A,T)=(A^3-1)^2-A^3T^2(1+T)^{-1}.
\]
Since
\[
T^2(1+T)^{-1}=(z-1)^2z^{-1}=z+z^{-1}-2,
\]
we obtain
\[
\Fcal_X(A,T)=A^6-A^3(z+z^{-1})+1.
\]
Therefore
\[
\Fcal_X(A,T)=(A^3-z)(A^3-z^{-1}).
\]

Let
\[
N_n=3\cdot 2^n.
\]
At level $n$, the graph is a cycle of length $N_n$. The Grover transition
matrix $U_n$ acts on the directed edges, and hence has size $2N_n$.

\begin{proposition}\label{prop:K3-resultant}
For the $K_3$-tower above, one has
\[
\det(AI_{2N_n}-U_n)=(A^{N_n}-1)^2.
\]
Consequently, for every monic polynomial $P(A)\in K[A]$,
\[
\det P(U_n)=\Res_A(A^{N_n}-1,P(A))^2.
\]
\end{proposition}

\begin{proof}
By the factorization
\[
\Fcal_X(A,T)=(A^3-z)(A^3-z^{-1}),
\]
the universal spectral factorization gives
\[
\det(AI_{2N_n}-U_n)
=
\prod_{\zeta\in W_n}(A^3-\zeta)(A^3-\zeta^{-1}).
\]
Here the Bass factor does not contribute because $\chi(X)=0$. Since
\[
\prod_{\zeta\in W_n}(A^3-\zeta)
=
A^{3\cdot2^n}-1=A^{N_n}-1,
\]
and the same identity holds with $\zeta^{-1}$ in place of $\zeta$, we
obtain
\[
\det(AI_{2N_n}-U_n)=(A^{N_n}-1)^2.
\]
Taking the resultant with $P(A)$, using the convention of
Section~\ref{subsec:resultants}, gives the assertion for $\det P(U_n)$.
\end{proof}

\begin{corollary}[The linear case]\label{cor:K3-linear}
For every \(a\in K\), one has
\[
\det(aI_{2N_n}-U_n)=(a^{N_n}-1)^2.
\]
Moreover, the following are equivalent:
\[
a^3\notin W
\]
and
\[
a\notin \bigcup_{n\ge0}\Spec(U_n).
\]
Equivalently, \(a\) is a non-eigenvalue at every finite level if and only
if \(a^3\) is not a \(2\)-power root of unity.
\end{corollary}

\begin{proof}
Apply Proposition~\ref{prop:K3-resultant} to \(P(A)=A-a\).  This gives
the determinant identity for every \(a\in K\).

Since \(N_n=3\cdot2^n\), the equality
\[
a^{N_n}=1
\]
is equivalent to
\[
(a^3)^{2^n}=1,
\]
that is, to \(a^3\in W_n\).  Hence \(a\) occurs as an eigenvalue at some
finite level if and only if \(a^3\in W\).  This proves the equivalence.
\end{proof}

\begin{remark}
This recovers the fixed non-eigenvalue computation
\[
\det(aI_{2N_n}-U_n)=(a^{3\cdot2^n}-1)^2
\]
from the spectral resultant formula.
\end{remark}

\begin{corollary}[Cyclotomic spectra]
Let
\[
P(A)=\Phi_m(A)
\]
be the $m$-th cyclotomic polynomial. If $m>2$, then
\[
\Phi_m\text{-spectrum occurs in }U_n
\quad\Longleftrightarrow\quad
m\mid N_n=3\cdot2^n.
\]
Moreover, if $m\mid N_n$, then
\[
m_n(\Phi_m)=2\varphi(m),
\]
and otherwise
\[
m_n(\Phi_m)=0.
\]
\end{corollary}

\begin{proof}
For $m>2$, the polynomial $\Phi_m(A)$ is coprime to $A^2-1$. Thus the
clean exceptional-zero criterion applies. By Proposition~\ref{prop:K3-resultant},
\[
\det \Phi_m(U_n)=\Res_A(A^{N_n}-1,\Phi_m(A))^2.
\]
Therefore $\Phi_m$-spectrum occurs in $U_n$ if and only if $A^{N_n}-1$
and $\Phi_m(A)$ have a common root. This happens if and only if a
primitive $m$-th root of unity is an $N_n$-th root of unity, equivalently
$m\mid N_n$. When this condition holds, each primitive $m$-th root
contributes once to $A^{N_n}-1$, and the square gives multiplicity $2$.
Hence $m_n(\Phi_m)=2\varphi(m)$.
\end{proof}

\begin{remark}
The cases $m=1$ and $m=2$ involve the factors $A-1$ and $A+1$, which meet
the Bass factor $A^2-1$. They are governed by the full multiplicity formula
rather than the clean coprime criterion.
\end{remark}

\begin{corollary}[Exceptional zeros in the $K_3$-tower]
Let $P(A)\in K[A]$ be monic. Then
\[
\Rcal_{X,P}(\zeta-1)=0
\]
for $\zeta\in W$ if and only if there exists a root $\theta$
of $P$ in a splitting field such that
\[
\theta^3=\zeta
\qquad\text{or}\qquad
\theta^3=\zeta^{-1}.
\]
Equivalently, the exceptional zero set is
\[
\mathcal Z_{X,P}
:=
\{\theta^3,\theta^{-3}: P(\theta)=0,\ \theta^3\in W\}.
\]
\end{corollary}

\begin{proof}
Since
\[
\Fcal_X(A,T)=(A^3-z)(A^3-z^{-1}),
\]
we have
\[
\Rcal_{X,P}(T)
=
\Res_A(A^3-z,P(A))\Res_A(A^3-z^{-1},P(A)).
\]
Thus $\Rcal_{X,P}(\zeta-1)=0$ if and only if $P(A)$ has a root $\theta$
satisfying $\theta^3=\zeta$ or $\theta^3=\zeta^{-1}$.
\end{proof}

\subsection{A two-variable example}\label{subsec:two-variable-example}

We next record a small example with $d=2$.  Let $p=2$ and let $X$ be the
one-vertex symmetric digraph with three inverse pairs of loops
$s_0,\bar s_0$, $s_1,\bar s_1$, and $s_2,\bar s_2$.  The Grover weight is
\[
w(e)=\frac{2}{6}=\frac13\in\mathbb Z_2.
\]
Let $\Gamma=\mathbb Z_2^2$ and define
\[
\alpha(s_0)=(1,0),\qquad
\alpha(s_1)=(0,1),\qquad
\alpha(s_2)=(0,0),
\]
with inverse values on the inverse edges.  Put $z_i=1+T_i$.  Then
\[
W_{\tau,\alpha}=\frac13(z_1+z_1^{-1}+z_2+z_2^{-1}+2),
\]
and hence
\[
\Fcal_X(A,T)=A^2+1-\frac{A}{3}(z_1+z_1^{-1}+z_2+z_2^{-1}+2).
\]
For the linear packet $P(A)=A-3$, we obtain
\[
\Rcal_{X,P}(T)=\Fcal_X(3,T)
=4-\frac{T_1^2}{1+T_1}-\frac{T_2^2}{1+T_2}.
\]
Modulo $2$,
\[
\overline{\Rcal}_{X,P}(T)=
\frac{T_1^2}{1+T_1}+\frac{T_2^2}{1+T_2}.
\]
Multiplying by the unit $(1+T_1)(1+T_2)$ gives
\[
T_1^2(1+T_2)+T_2^2(1+T_1)
=(T_1+T_2)((1+T_1)(1+T_2)-1).
\]
Equivalently, up to multiplication by a unit,
\[
\overline{\Rcal}_{X,P}(T)=
\bigl((1+T_1)(1+T_2)^{-1}-1\bigr)
\bigl((1+T_1)(1+T_2)-1\bigr).
\]
The two displayed factors are distinguished linear primes of the form
$(\sigma-1)$. Hence the divisor-sum definition gives
\[
\mu(\Rcal_{X,P})=0,
\qquad
\lambda(\Rcal_{X,P})=2.
\]
The torsion non-vanishing condition also holds.  Indeed, after embedding the relevant cyclotomic field into $\mathbb C$,
if $z_1,z_2$ are roots of unity and
$\Fcal_X(3,z_1-1,z_2-1)=0$, then
\[
z_1+z_1^{-1}+z_2+z_2^{-1}=8,
\]
which is impossible because the absolute value of the left-hand side is
at most $4$.  Since
\[
\chi(X)=1-3=-2,
\qquad
v_2(\Res_A(A^2-1,A-3))=v_2(8)=3,
\]
Theorem~\ref{thm:resultant-iwasawa} gives
\[
v_2(\det(3I-U_n))
=
6\cdot 2^{2n}+2n\cdot 2^n+O(2^n).
\]
This example exhibits the two-variable leading term in the growth formula.

\subsection{A two-variable equivariant Kida example}\label{subsec:two-variable-kida-example}

We keep the two-variable base graph of Section~\ref{subsec:two-variable-example}. Let
\[
G=\mathbb Z/2\mathbb Z=\{0,1\}
\]
and define a voltage assignment \(\beta:E(X)\to G\) by
\[
\beta(s_0)=1,
\qquad
\beta(s_1)=\beta(s_2)=0,
\qquad
\beta(\bar s_i)=\beta(s_i).
\]
Since \(\beta(s_0)\) generates \(G\), the derived cover
\(Y=X(G,\beta)\to X\) is connected and \([Y:X]=2\). Let \(P(A)=A-3\), and
let \(U_n^Y\) denote the Grover transition matrix on the \(n\)-th layer of
the pullback tower over \(Y\).

\begin{proposition}[The unramified Kida formula in two variables]
\label{prop:two-variable-kida}
In the two-variable cover above,
\[
\mu(\Rcal_{Y,P})=0,
\qquad
\lambda(\Rcal_{Y,P})=4=2\lambda(\Rcal_{X,P}).
\]
Moreover,
\[
\mu^{\mathrm{qw}}_{Y,P}=12=2\mu^{\mathrm{qw}}_{X,P},
\]
and the Grover determinant on the pullback tower over \(Y\) satisfies
\[
v_2\bigl(\det(3I-U_n^Y)\bigr)
=
12\cdot 2^{2n}+4n\cdot 2^n+O(2^n).
\]
\end{proposition}

\begin{proof}
The group \(G\) has two irreducible representations: the trivial
representation \(\rho_0\) and the sign representation \(\rho_1\).  Put
\[
u=\frac{T_1^2}{1+T_1},
\qquad
v=\frac{T_2^2}{1+T_2}.
\]
For the trivial representation, \(\Rcal_{X,P,\rho_0}=\Rcal_{X,P}=4-u-v\).
For the sign representation, the contribution of the pair
\(s_0,\bar s_0\) changes sign, and hence
\[
\Fcal_{X,\rho_1}(A,T)
=
A^2+1-\frac{A}{3}\bigl(-z_1-z_1^{-1}+z_2+z_2^{-1}+2\bigr).
\]
Therefore
\[
\Rcal_{X,P,\rho_1}(T)=\Fcal_{X,\rho_1}(3,T)=8+u-v.
\]
This is not equal to \(\Rcal_{X,P}\) in characteristic zero, but modulo
\(2\) both reduce to \(u+v\). By Proposition~\ref{prop:equiv-factor},
\[
\Rcal_{Y,P}=(4-u-v)(8+u-v)=32-4u-12v-u^2+v^2.
\]
Modulo \(2\), this gives
\[
\overline{\Rcal}_{Y,P}=u^2+v^2=(u+v)^2
=\overline{\Rcal}_{X,P}^{\,2}.
\]
Since \(\overline{\Rcal}_{X,P}\ne0\), we have \(\Rcal_{Y,P}\ne0\).  The
base computation in Section~\ref{subsec:two-variable-example} gives
\(\mu(\Rcal_{X,P})=0\) and \(\lambda(\Rcal_{X,P})=2\), and the last
display gives
\[
\mu(\Rcal_{Y,P})=0,
\qquad
\lambda(\Rcal_{Y,P})=4.
\]
This agrees with Theorem~\ref{thm:kida}.

The torsion non-vanishing condition for the cover also holds.  Indeed, the
factor \(4-u-v\) was already treated in Section~\ref{subsec:two-variable-example}. If the factor
\(8+u-v\) vanished at roots of unity, then, after embedding the relevant
cyclotomic field into \(\mathbb C\), we would have
\[
z_1+z_1^{-1}-z_2-z_2^{-1}=-8,
\]
which is impossible because the absolute value of the left-hand side is at
most \(4\).  Thus Theorem~\ref{thm:resultant-iwasawa} applies to the pullback
tower over \(Y\). Since \(\chi(Y)=2\chi(X)=-4\) and
\(v_2(\Res_A(A^2-1,A-3))=v_2(8)=3\), the Bass-corrected invariant is
\[
\mu^{\mathrm{qw}}_{Y,P}=0-(-4)\cdot3=12.
\]
Together with \(\lambda(\Rcal_{Y,P})=4\), this gives the stated growth
formula.
\end{proof}

\subsection{A non-abelian 5-group application}\label{subsec:heisenberg-application}

We now give a concrete application of the unramified equivariant quantum Kida
formula in a genuinely non-abelian setting.  The example shows that the
integrality and nonzero-resultant hypotheses of Theorem~\ref{thm:kida}
are satisfied for a non-abelian $5$-group cover, and that the theorem
computes the corresponding Iwasawa invariants.

We first treat the linear polynomial $P(A)=A-6$.  We then consider the
quadratic spectral packet
\[
P(A)=(A-6)(A-11),
\]
which illustrates the polynomial nature of the spectral resultant in the
Kida formula.  Finally, we exhibit a polynomial $P_5(A)$ with torsion zeros, giving a
concrete instance of the exceptional-zero criterion of
Section~\ref{sec:exceptional}.

We consider three packets.  The first two lie in the torsion non-vanishing
regime, while the third lies in the torsion-zero regime.

Let $p=5$. Let $X$ be the one-vertex symmetric digraph with three inverse
pairs of loops
\[
s_0,\bar s_0,\qquad s_1,\bar s_1,\qquad s_2,\bar s_2.
\]
Then $d_v=6$, and the Grover weight is
\[
w(e)=\frac{2}{6}=\frac13\in\mathbb Z_5.
\]
Let $\Gamma=\mathbb Z_5$. Define a voltage assignment
$\alpha:E(X)\to\Gamma$ by
\[
\alpha(s_0)=1,
\qquad
\alpha(s_1)=\alpha(s_2)=0,
\]
and $\alpha(\bar s_i)=-\alpha(s_i)$. Put $z=1+T$. Since $X$ has one
vertex, the universal weighted adjacency matrix is a scalar:
\[
W_{\tau,\alpha}=\frac13(z+z^{-1}+4).
\]
Therefore
\[
\Fcal_X(A,T)=A^2+1-\frac{A}{3}(z+z^{-1}+4).
\]

Let $P(A)=A-6$. Then
\[
\Rcal_{X,P}(T)=\Fcal_X(6,T).
\]
We compute
\[
\Fcal_X(6,T)=6^2+1-\frac63(z+z^{-1}+4).
\]
Hence
\[
\Rcal_{X,P}(T)=37-2(z+z^{-1}+4)=29-2(z+z^{-1}).
\]
Since
\[
z+z^{-1}=2+\frac{T^2}{1+T},
\]
we get
\[
\Rcal_{X,P}(T)=25-\frac{2T^2}{1+T}.
\]
Modulo $5$,
\[
\overline{\Rcal}_{X,P}(T)=-\frac{2T^2}{1+T}.
\]
Since $1+T$ is a unit in $\mathbb F_5[[T]]$, we have
\[
\lambda_k(\overline{\Rcal}_{X,P})
=
\ord_{(T)}\left(-\frac{2T^2}{1+T}\right)
=2.
\]
Therefore
\[
\mu(\Rcal_{X,P})=0,
\qquad
\lambda(\Rcal_{X,P})=2.
\]

Let $G=H_5(\mathbb F_5)$ be the Heisenberg group of order $5^3=125$, with
presentation
\[
G=\langle x,y,c\mid x^5=y^5=c^5=1,\ c=[x,y],\ c\text{ central}\rangle.
\]
The representation-theoretic decomposition is genuinely non-abelian.  The
abelianization of $G$ has order $25$, giving $25$ one-dimensional irreducible
representations, and the remaining irreducible representations have dimension
$5$.  Since the number of nontrivial central characters is $4$, there are four
such $5$-dimensional representations.  Thus
\[
25\cdot 1^2+4\cdot 5^2=125=|G|.
\]
This is the regular-representation dimension identity used in
Proposition~\ref{prop:equiv-factor}.

Let $K/\mathbb Q_5$ be a splitting field for $G$, with valuation ring
$\mathcal O$.  The elements $1/3$, $6$, and $11$ lie in
$\mathbb Z_5\subset\mathcal O$. Hence the integrality assumptions in
Theorem~\ref{thm:kida} are satisfied for the linear and quadratic packets
below.

Define a $G$-cover $Y\to X$ by the voltage assignment
\[
\beta(s_0)=1_G,
\qquad
\beta(s_1)=x,
\qquad
\beta(s_2)=y,
\]
and $\beta(\bar s_i)=\beta(s_i)^{-1}$. Since $x$ and $y$ generate $G$, the
cover is connected. Therefore $[Y:X]=|G|=125$.
Here $\beta$ defines the finite Heisenberg cover $Y\to X$, whereas
$\alpha$ defines the $\mathbb Z_5$-tower. The tower over $Y$ is the
pullback tower defined by $\alpha\circ\pi_E$.

\begin{proposition}
For the Heisenberg $5$-group cover above and $P(A)=A-6$, one has
\[
\mu(\Rcal_{Y,P})=0
\]
and
\[
\lambda(\Rcal_{Y,P})=250.
\]
In particular,
\[
\lambda(\Rcal_{Y,P})=[Y:X]\lambda(\Rcal_{X,P}).
\]
\end{proposition}

\begin{proof}
We have already computed
\[
\mu(\Rcal_{X,P})=0,
\qquad
\lambda(\Rcal_{X,P})=2.
\]
The field $K$ was chosen to split $G$, the Grover weight satisfies
$w(e)=1/3\in\mathbb Z_5\subset\mathcal O$, and
$P(A)=A-6\in\mathbb Z_5[A]\subset\mathcal O[A]$. Thus the integrality and
splitting-field hypotheses of Theorem~\ref{thm:kida} are satisfied.
Modulo $\mathfrak m$, Proposition~\ref{prop:equiv-factor} together with
Lemma~\ref{lem:pgroup-reduction} gives
\[
\overline{\Rcal}_{Y,P}(T)
=
\overline{\Rcal}_{X,P}(T)^{|G|}
=
\left(-\frac{2T^2}{1+T}\right)^{125}.
\]
In particular, $\overline{\Rcal}_{Y,P}\ne0$, and hence
$\Rcal_{Y,P}\ne0$.  We may therefore apply Theorem~\ref{thm:kida}. It gives
\[
\mu(\Rcal_{Y,P})=0,
\qquad
\lambda(\Rcal_{Y,P})=125\cdot2=250.
\]
Since $[Y:X]=125$, the final equality follows.
\end{proof}

\begin{remark}
In this example,
\[
\Res_A(A^2-1,P)=6^2-1=35,
\]
and hence
\[
v_5(\Res_A(A^2-1,P))=1.
\]
The base graph $X$ has one vertex and three undirected loops, so
\[
\chi(X)=1-3=-2.
\]
Thus
\[
\mu^{\mathrm{qw}}_{X,P}
=
\mu(\Rcal_{X,P})-
\chi(X)v_5(\Res_A(A^2-1,P))
=
0-(-2)\cdot1=2.
\]
Since $\chi(Y)=[Y:X]\chi(X)=125\cdot(-2)=-250$, we also have
\[
\mu^{\mathrm{qw}}_{Y,P}=0-(-250)\cdot1=250.
\]
Consequently,
\[
\mu^{\mathrm{qw}}_{Y,P}=[Y:X]\mu^{\mathrm{qw}}_{X,P}.
\]

\end{remark}

\begin{proposition}[A quadratic spectral packet]\label{prop:heisenberg-quadratic}
In the Heisenberg $5$-group example, put
\[
P(A)=(A-6)(A-11).
\]
Then
\[
\mu(\Rcal_{X,P})=0,
\qquad
\lambda(\Rcal_{X,P})=4.
\]
Consequently, for the Heisenberg cover $Y\to X$,
\[
\mu(\Rcal_{Y,P})=0,
\qquad
\lambda(\Rcal_{Y,P})=500.
\]
\end{proposition}

\begin{proof}
From
\[
z+z^{-1}+4=6+\frac{T^2}{1+T},
\]
we may rewrite the spectral element as
\[
\Fcal_X(A,T)
=A^2+1-\frac{A}{3}(z+z^{-1}+4)
=(A-1)^2-\frac{A}{3}\frac{T^2}{1+T}.
\]
By the root expression for the resultant,
\[
\Rcal_{X,P}(T)=\Fcal_X(6,T)\Fcal_X(11,T).
\]
We have
\[
\Fcal_X(6,T)=25-2\frac{T^2}{1+T},
\qquad
\Fcal_X(11,T)=100-\frac{11}{3}\frac{T^2}{1+T}.
\]
Modulo $5$, since $11/3\equiv 2$, we get
\[
\overline{\Fcal_X(6,T)}=-2\frac{T^2}{1+T},
\qquad
\overline{\Fcal_X(11,T)}=-2\frac{T^2}{1+T}.
\]
Therefore
\[
\overline{\Rcal}_{X,P}(T)
=4\frac{T^4}{(1+T)^2}.
\]
Since $1+T$ is a unit in $\mathbb F_5[[T]]$, we obtain
\[
\mu(\Rcal_{X,P})=0,
\qquad
\lambda(\Rcal_{X,P})=4.
\]
The integrality and splitting-field hypotheses are the same as in the
linear case. Moreover, the reduction identity from
Proposition~\ref{prop:equiv-factor} and Lemma~\ref{lem:pgroup-reduction}
gives
\[
\overline{\Rcal}_{Y,P}
=
\overline{\Rcal}_{X,P}^{125}
\ne0,
\]
so $\Rcal_{Y,P}\ne0$. Applying Theorem~\ref{thm:kida}, we get
\[
\lambda(\Rcal_{Y,P})=125\cdot4=500,
\]
and $\mu(\Rcal_{Y,P})=0$.
\end{proof}

\begin{remark}
For this quadratic spectral packet,
\[
\Res_A(A^2-1,P)=P(1)P(-1)=50\cdot84,
\]
so $v_5(\Res_A(A^2-1,P))=2$. Since $\chi(X)=-2$,
\[
\mu^{\mathrm{qw}}_{X,P}=0-(-2)\cdot2=4.
\]
Similarly, $\chi(Y)=125\chi(X)=-250$, and hence
\[
\mu^{\mathrm{qw}}_{Y,P}=0-(-250)\cdot2=500
=[Y:X]\mu^{\mathrm{qw}}_{X,P}.
\]
This example also shows that torsion non-vanishing may hold beyond the
unit criterion of Proposition~\ref{prop:unit-criterion}. For instance,
$\Rcal_{X,A-6}(0)=25\notin\mathbb Z_5^\times$, but a torsion zero would
force
\[
\zeta+\zeta^{-1}=\frac{29}{2},
\]
which is impossible after embedding the corresponding cyclotomic field into
\(\mathbb C\). The factor \(A-11\) is similar and would force
\[
\zeta+\zeta^{-1}=\frac{322}{11}.
\]
Thus the linear and quadratic packets both satisfy torsion non-vanishing.
\end{remark}

\begin{proposition}[A torsion-zero spectral packet]\label{prop:p5-torsion-packet}
In the one-vertex $6$-regular base graph of this subsection, put
\[
P_5(A)
=
A^4-\frac73A^3+\frac{29}{9}A^2-\frac73A+1
\in\mathbb Z_5[A].
\]
Then
\[
P_5(1)=\frac59,
\qquad
P_5(-1)=\frac{89}{9},
\]
so $P_5$ is coprime to $A^2-1$.  Moreover,
\[
\Rcal_{X,P_5}(T)
=
\frac{\Phi_5(z)^2}{81z^4},
\qquad z=1+T.
\]
As a power series over $\mathbb Z_5$, it satisfies
\[
\mu(\Rcal_{X,P_5})=0,
\qquad
\lambda(\Rcal_{X,P_5})=8.
\]
Consequently, for $\zeta\in\mu_{5^\infty}$,
\[
\Rcal_{X,P_5}(\zeta-1)=0
\]
if and only if $\zeta$ is a primitive fifth root of unity. Thus the torsion
non-vanishing hypothesis in Theorem~\ref{thm:resultant-iwasawa} fails. The
roots of $P_5$ occur as Grover eigenvalues on every level containing the
primitive fifth characters, in particular for every $n\ge1$.
\end{proposition}

\begin{proof}
Let
\[
s=z+z^{-1},
\qquad
x=A+A^{-1}.
\]
Since the constant term of $P_5$ is $1$, every root of $P_5$ is nonzero.
Dividing $P_5(A)$ by $A^2$, we find
\[
A^{-2}P_5(A)
=
(A+A^{-1})^2-\frac73(A+A^{-1})+\frac{11}{9}.
\]
Thus the corresponding values of $x=A+A^{-1}$ are the roots of
\[
x^2-\frac73x+\frac{11}{9}=0.
\]
For the one-vertex graph considered above,
\[
\Fcal_X(A,T)=A^2+1-\frac{A}{3}(z+z^{-1}+4).
\]
Put
\[
c=\frac{s+4}{3}.
\]
Then
\[
\Fcal_X(A,T)=A(A+A^{-1}-c)=A(x-c)
\]
for $A\ne0$. Since $P_5$ is reciprocal, its roots occur in pairs
$\theta,\theta^{-1}$ with the same value of
$x=\theta+\theta^{-1}$. For such a pair,
\[
\Fcal_X(\theta,T)\Fcal_X(\theta^{-1},T)=(x-c)^2.
\]
Hence
\[
\Rcal_{X,P_5}(T)
=
\left(c^2-\frac73c+\frac{11}{9}\right)^2
=
\left(\frac{s^2+s-1}{9}\right)^2.
\]
Since
\[
s^2+s-1=(z+z^{-1})^2+(z+z^{-1})-1
=\frac{z^4+z^3+z^2+z+1}{z^2},
\]
we obtain
\[
\Rcal_{X,P_5}(T)
=
\frac{(z^4+z^3+z^2+z+1)^2}{81z^4}
=
\frac{\Phi_5(z)^2}{81z^4}.
\]
Modulo $5$, we have
\[
\Phi_5(1+T)=(1+T)^4+(1+T)^3+(1+T)^2+(1+T)+1\equiv T^4.
\]
Since $81\equiv1\pmod 5$ and $z=1+T$ is a unit in $\mathbb F_5[[T]]$, the
reduction of the resultant is
\[
\overline{\Rcal}_{X,P_5}(T)=\frac{T^8}{(1+T)^4}=T^8(1+T)^{-4}.
\]
Since $(1+T)^{-4}$ is a unit in $\mathbb F_5[[T]]$, the order at $T=0$ is $8$. Therefore
\[
\mu(\Rcal_{X,P_5})=0,
\qquad
\lambda(\Rcal_{X,P_5})=8.
\]
Thus the resultant vanishes exactly when $\Phi_5(z)=0$, that is, when
$z$ is a primitive fifth root of unity. Since $P_5(1)P_5(-1)\ne0$, the
Bass factor does not contribute. The exceptional-zero statement now follows
from Theorem~\ref{thm:exceptional}.
\end{proof}

\begin{remark}[Summary of the Heisenberg packets]
The linear and quadratic packets lie in the torsion non-vanishing regime,
whereas $P_5$ lies in the torsion-zero regime. For $P_5$, the reduction
formula gives
\[
\overline{\Rcal}_{Y,P_5}
=
\overline{\Rcal}_{X,P_5}^{125}\ne0,
\]
so \(\Rcal_{Y,P_5}\ne0\) and the algebraic Kida formula applies to this packet as well.  It gives
\[
\lambda(\Rcal_{Y,P_5})=125\cdot 8=1000.
\]
The corresponding finite-level spectral behavior is governed by
Theorem~\ref{thm:exceptional}.

All three packets considered in this subsection satisfy
\(\mu(\Rcal_{X,P})=0\). Their algebraic \(\lambda\)-invariants and
torsion behavior are summarized as follows:
\[
\begin{array}{c|c|c|c}
P(A) & \lambda(\Rcal_{X,P}) & \lambda(\Rcal_{Y,P}) & \text{torsion behavior}\\
\hline
A-6 & 2 & 250 & \text{torsion non-vanishing}\\
(A-6)(A-11) & 4 & 500 & \text{torsion non-vanishing}\\
P_5 & 8 & 1000 & \text{torsion zeros at primitive fifth roots}\\
\end{array}
\]
The entries in the \(\lambda(\Rcal_{Y,P})\)-column are algebraic
Kida-scaled invariants supplied by Theorem~\ref{thm:kida}. For the last
row, the relevant finite-level interpretation is given by
Theorem~\ref{thm:exceptional}.
\end{remark}

\section*{Funding}
This work was supported by JSPS KAKENHI Grant Number 23K12507.

\section*{Acknowledgements}

The authors thank Professor Iwao Sato for helpful comments on the
Grover--Ihara determinant formalism and weighted graph zeta functions.

\section*{Statements and Declarations}

\paragraph{Competing interests.}
The authors have no competing interests to declare.

\paragraph{Data availability.}
Data sharing is not applicable to this article as no datasets were
generated or analyzed.

\bibliographystyle{abbrvurl}
\bibliography{references}

\end{document}